\documentclass{amsart}

\usepackage{mathptmx}
\usepackage{helvet}
\usepackage{courier}
\usepackage{graphicx}
\usepackage{multicol}
\usepackage{footmisc}
\usepackage[T1]{fontenc}
\usepackage{euscript}
\usepackage{eufrak}
\usepackage{epic}

\usepackage{amsfonts,amsthm,amsmath}
\usepackage{amssymb}
\usepackage{amscd}
\usepackage{enumerate}
\usepackage[usenames,dvipsnames]{color}
\usepackage{tikz}
\usepackage{soul}
\usepackage{array}
\usepackage{array,tabularx}

\numberwithin{equation}{section}
\numberwithin{equation}{subsection}

\theoremstyle{plain}

\newtheorem{theorem}[equation]{Theorem}
\newtheorem{lemma}[equation]{Lemma}
\newtheorem{proposition}[equation]{Proposition}

\newtheorem{corollary}[equation]{Corollary}

\theoremstyle{definition}\newtheorem{notation}[equation]{Notation}
\newtheorem{example}[equation]{Example}
\newtheorem{remark}[equation]{Remark}

\newtheorem{definition}[equation]{Definition}

\newtheorem{bekezdes}[equation]{}

\newcommand{\chic}{\mathfrak{r}}

\newcommand{\bZ}{{\mathbb Z}}

\newcommand{\bP}{{\mathbb P}}

\newcommand{\calC}{{\mathcal C}}
\newcommand{\cJ}{{\mathcal J}}

\newcommand{\cV}{{\mathcal V}}

\newcommand{\calS}{{\mathcal S}}

\newcommand{\cI}{{\mathcal I}}

\newcommand{\cO}{{\mathcal O}}\newcommand{\calO}{{\mathcal O}}
\newcommand{\cF}{{\mathcal F}}

\newcommand{\calL}{{\mathcal L}}

\newcommand{\calQ}{{\mathcal Q}}

\newcommand{\tX}{\widetilde{X}}

\newcommand{\C}{{\calc}}

\newcommand{\rank}{{\rm rank}\, }

\def\blfootnote{\xdef\@thefnmark{}\@footnotetext}

\newcommand{\bt}{{\bf t}}

\newcommand{\bH}{{\mathbb H}}

\newcommand{\cali}{{\mathcal I}}

\newcommand{\calF}{{\mathcal F}}

\newcommand{\hh}{\mathfrak{h}}
\newcommand{\pp}{\mathfrak{p}}
\newcommand{\RR}{\mathfrak{R}}

\newcommand{\calv}{\mathcal{V}}

\newcommand{\Z}{\mathbb{Z}}
\newcommand{\Q}{\mathbb{Q}}
\newcommand{\R}{\mathbb{R}}

\begingroup\makeatletter\ifx\SetFigFontNFSS\undefined%
\gdef\SetFigFontNFSS#1#2#3#4#5{%
  \reset@font\fontsize{#1}{#2pt}%
  \fontfamily{#3}\fontseries{#4}\fontshape{#5}%
  \selectfont}%
\fi\endgroup%

\newcommand{\tx}{\tilde{X}}

\newcommand{\calt}{{\mathcal T}}

\def\C{\mathbb C}
\def\Q{\mathbb Q}
\def\R{\mathbb R}
\def\bH{\mathbb H}

\def\Z{\mathbb Z}

\newcommand{\cale}{{\mathcal E}}

\author{Tam\'as \'Agoston}
\address{Alfr\'ed R\'enyi Institute of Mathematics,
Hungarian Academy of Sciences,
Re\'altanoda utca 13-15, H-1053, Budapest, Hungary
\newline
 \hspace*{4mm} ELTE - University of Budapest, Dept. of Geometry, Budapest, Hungary
 }
\email{agoston.tamas@renyi.hu}

\author{Andr\'as N\'emethi}
\address{Alfr\'ed R\'enyi Institute of Mathematics,
Hungarian Academy of Sciences,
Re\'altanoda utca 13-15, H-1053, Budapest, Hungary \newline
 \hspace*{4mm} ELTE - University of Budapest, Dept. of Geometry, Budapest, Hungary \newline \hspace*{4mm}
BCAM - Basque Center for Applied Math.,
Mazarredo, 14 E48009 Bilbao, Basque Country – Spain}
\email{nemethi.andras@renyi.hu }

\title{The analytic lattice cohomology of isolated  singularities
}

\begin{document}

\thanks{The  authors are partially supported by NKFIH Grant ``\'Elvonal (Frontier)'' KKP 126683.}

\begin{abstract}
We associate (under a minor assumption)
to any analytic isolated singularity of dimension $n\geq 2$ the `analytic
lattice cohomology' $\bH^*_{an}=\oplus_{q\geq 0}\bH^q_{an}$. Each $\bH^q_{an}$
is  a graded ${\mathbb Z}[U]$--module. It   is the extension to higher dimension
of the
`analytic lattice cohomology' defined for a normal surface singularity with a rational homology sphere link. This latest one is the analytic analogue of the `topological lattice cohomology' of the link of the  normal surface singularity, which conjecturally is isomorphic to the Heegaard Floer cohomology of the link.

The definition uses a good resolution $\tX$ of the singularity $(X,o)$. Then we prove the independence of the choice of the  resolution, and  we show that
the Euler characteristic of $\bH^*_{an}$   is $h^{n-1}(\calO_{\tX})$.
In the case of a hypersurface weighted homogeneous singularity we relate it to the Hodge
spectral numbers of the first interval.
\end{abstract}

\maketitle


\date{}

\pagestyle{myheadings} \markboth{{\normalsize  T. \'Agoston , A. N\'emethi}} {{\normalsize Analytic lattice cohomology }}


\section{Introduction}\label{s:intr}

\subsection{} In the classification of singular germs  one can
proceed in many different directions.
The first level is the topological classification of the singularity links
using topological (smooth) invariants. Then one continues with the much harder
analytic classification with the help of different analytic invariants.
In this process one usually   uses sheaf cohomologies associated with different analytic sheaves.
 However, if we wish to keep certain deep  interference with recent developments in topology  then we might naturally  ask:

{\em what are the analytic analogs of the celebrated  cohomology theories produced by the low-dimensional \hspace*{3mm} topology in the last decades (e.g. of the Heegaard Floer theory)?}

The Heegaard Floer theory,
defined by Ozsv\'ath and Szab\'o,
 associates to any oriented compact 3-manifold a graded $\Z[U]$--module, see e.g.
 \cite{OSz,OSz7,OSzP}.
Its Euler characteristic is the Seiberg--Witten invariant of the link.
It is equivalent with several other cohomology theories:  the
{\em  Monopole Floer Homology} of Kronheimer and Mrowka, the
 {\em Seiberg--Witten version of Floer homology} presented by Marcolli and Wang, or  Hutchings'
 {\em Embedded  Contact Homology}.
They produce extremely strong  results in low dimensional topology.
Our task is to develop  an  {\it analytic analogue}.

\bekezdes\label{1.1.1}
 The first bridge between the Heegaard Floer theory
  and the analytic theory of singularities is
  realized by the {\it topological lattice cohomology}  $\bH^*_{top}=\oplus_{q\geq 0}\bH^q_{top}$
   introduced by the second author  in \cite{Nlattice}. It is associated with the link of a normal surface singularity (a special
  plumbed 3--manifold), whenever the link is a rational homology sphere.
  Each $\bH^q_{top}$ is a  graded $\Z[U]$--module.
An improvement of $\bH^0_{top}$ is a {\it
graded root}, a special tree with $\bZ$--graded vertices (where the edges correspond to the $U$--action).
 They were  defined using a good resolution.
 Some of their   key properties are the following:

(a)  $\bH^*_{top}$ is independent of the choice of the resolution, it depends only on the link $M$,

(b) $\bH^*_{top}(M)=
 \oplus_{\sigma\in{\rm Spin}^c(M)}\bH^*_{top}(M,\sigma)$,

(c) the Euler characteristic is the (normalized) Seiberg--Witten invariant indexed by
 ${\rm Spin}^c(M)$, 

(d)  $\bH^*_{top}(M)$
  satisfies  several exact sequences (analogues of the exact triangles of
 $HF^+$) \cite{Greene,Nexseq}.

(e) Conjecturally $\bH^*_{top}$
 is isomorphic with  Heegaard Floer cohomology $HF^+$
  for links of  normal surface singularities which are rational homology spheres \cite{Nlattice}.
  More  precisely, one expects
  $HF^+_{odd/even}\simeq \oplus_{q\ odd/even} \,\bH^q_{top}$ as graded $\Z[U]$--modules (up to a shift).

\smallskip

This conjecture  has been affirmatively answered for a number of important families of singularities \cite{NOSz,OSSz3},
including those links which are Seifert fibered three-manifolds.
\cite{OSSz3} provides a spectral sequence from the lattice cohomology to the $HF$--cohomology, whose degeneration is
equivalent with the conjecture.

For several properties and applications  in singularity theory see \cite{NOSz,NSurgd,NGr,Nexseq,NeLO}. For its connection with the classification of
projective rational plane cuspidal curves (via superisolated surface  singularities) see \cite{NSurgd,BodNem,BodNem2,BLMN2,BCG,BL1}.
Furthermore, by its construction and key properties, $\bH^*_{top}$  realizes several  deep connections with analytic invariants of the germ   as well
(e.g. it provides sharp topological bounds for analytic invariants, see e.g.  \cite{NSig,NSigNN}).

\bekezdes Recently, in \cite{AgNe1,AgNeIII} we introduced their analytic analogues,
the analytic lattice cohomology $\bH^*_{an}=\oplus_{q\geq 0}\bH^q_{an}$, associated with
a normal surface singularity with a rational homology sphere link.
It is  constructed from analytic invariants of a good resolution, however
it turns out that it is independent of the choice of the resolution.
 Formally it has a  very similar structure as its topological analogue, e.g.
  the analogue of \ref{1.1.1}(b) is valid and each $\bH^*_{an}(X,\sigma)$ is a graded $\Z[U]$--module.
   The Euler characteristic of $\bH^*_{an}(X)$ is
the equivariant geometric genus.

Additionally, we succeeded in  constructing even a morphism of graded $\bZ[U]$-modules
$\mathfrak{H}^q:\bH^q  _{an}\to \bH^q_{top}$.
This is an isomorphism for some `nice' analytic structures.
In such cases we have the identity of the  Euler characteristics as well, namely of the Seiberg--Witten invariant with the
geometric genus. In fact, historically, this identity (The Seiberg--Witten Invariant Conjecture of the the second author  and Nicolaescu \cite{NeNi,NJEMS}, valid for `nice' analytic structures)
led to the discovery of $\bH^*_{top}$.
However, if we fix a topological type, and we move the possible analytic structure supported on this topological type, then the analytic lattice cohomologies reflect the modification of the analytic structures, for several examples see \cite{AgNe1}.

\bekezdes What is very surprising is that $\bH^*_{an}$ can be extended to other dimensions too,
to isolated singularities of dimension $n\geq2$, but even to the case of curves.
In this note we present {\it this extension to the higher dimensional singular germs}.

Again, in the definition we use a good resolution $\tX\to X$ of the singular germ $(X,o)$, with exceptional curve $E$. In the definition the multivariable divisorial filtration
associated with the irreducible components of $E$ has a key role.   We
 verify  that the newly defined $\bH^*_{an}$ is independent of the resolution whenever $h^{n-1}({\mathcal O}_E)=0$, and its Euler characteristic is
$h^{n-1}({\mathcal O}_{\widetilde{X}})$.

As an example, for isolated weighted homogeneous hypersurface singularities,
by the Reduction Theorem \ref{th:REDAN} the lattice cohomology can be computed via the divisorial
filtration of a unique
exceptional divisor (the exceptional divisor of the weighted blow up).
Since this $\Z$--filtration can be identified with the corresponding Newton filtration
(and the number of lattice points below the
Newton diagram is $p_g=h^{n-1}(\calO_{\tX})$, and the Hilbert function of the Newton filtration can be identified by the
Hodge spectrum in the interval $(0,1)$), we get a  characterization of $\bH^*_{an}$ by the
Hodge spectrum in $(0,1)$.

\bekezdes In fact, we obtain  more than the definition of the `lattice {\it cohomology}'.
Indeed, we define a sequence  of (finite cubical topological) spaces $\{S_n\}_{n\in\Z_{\geq  n_0}}$
with inclusions $\cdots \subset S_n\subset S_{n+1}\subset \cdots$ such that
$\bH^*_{an}=\oplus_{n\geq n_0}H^*(S_n, \Z)$, and the homotopy type of the tower of spaces depends only
on the analytic type of $(X,o)$. Therefore, in the spirit of the constructions of
`Khovanov homotopy type' of R. Lipshitz and S. Sarkar,
or of `Knot Floer stable homotopy type' of C. Manolescu and S. Sarkar,
in fact we have constructed the `(analytic) lattice homotopy type' of $(X,o)$ via the tower
$\{S_n\}_n$.

\bekezdes Note that in the case $n=2$ the analytic lattice cohomlogy was defined using as model
the topological lattice cohomology (and it was motivated by the topological cohomologies of the low-dimensional topology).

The higher dimensional case  has the  interesting aspect
that  we define the analytic lattice cohomology without having any  parallel topological
model.
 In fact,  the definition of the topological
$\bH_{top}^*$ is obstructed very seriously, since in the dimensions $n>2$ the link $M$
contains much less information from the singularity, e.g. $M$ can even be the standard sphere $S^{2n-1}$ for rather non-trivial analytic types $(X,o)$.

However, we expect the existence of a parallel theory in this higher dimensional case too: our expectation is
that it  should be the higher dimensional version of
the {\it Embedded Contact Homology} (ECH),  where the contact structure (induced by the analytic structure of $(X,o)$) on $M$
really plays a role.  (Recall that in the $n=2$ case this contact structure can  topologically be identified \cite{CNP},
a fact which does not hold in higher dimensions \cite{Ust}.)  Research in finding ECH in higher dimension was
initiated  by Colin--Honda  \cite{CH}.

\subsection{The structure of the paper}
In section 2  we review the general
definition of the lattice cohomology, of  the path lattice cohomology  and of the graded
root associated with a weight function.
In section 3 we review some statements regarding the Euler characteristic of a lattice cohomology
(in a combinatorial setup).

In section 4, after  we review certain analytic results regarding singularities and resolutions,
we define the analytic lattice cohomology and graded root (via a good resolution).
In Theorem \ref{th:annlattinda} we prove  their independence of the resolution. Using results of section 3 we
identify  the Euler characteristic as well. Subsection \ref{ss:anRT}
proves a `Reduction Theorem'. Using this we can reduce the rank of the lattice
(used in the basic construction).  This new lattice is identified by a set
of `bad' vertices.

In order to  define the new objects, and also to prove their  independence of the resolution,
we need to impose an assumption, namely the vanishing of $h^{n-1}(\calO_E)$. In section
5 we relate this vanishing with some  mixed Hodge theoretical
invariants. E.g, in the case of
isolated hypersurface singularities it is equivalent with the non-existence of spectral numbers
equal to one. (Hence, if the link is rational homology sphere, then this condition  is automatically satisfied.)
In section 6 we discuss the case of weighted homogeneous hypersurface singularities.

For more examples in the case $n=2$ see \cite{AgNe1}.

\section{Preliminaries. Basic properties of lattice cohomology}\label{s:Prem1}

\subsection{The lattice cohomology associated with  a weight function}\label{ss:latweight} \cite{NOSz,Nlattice}

\bekezdes {\bf Weight function.}
 We consider a free $\Z$-module, with a fixed basis
$\{E_v\}_{v\in\calv}$, denoted by $\Z^s$, $s:=|\calv|$.
Additionally, we consider a {\it weight  function} $w_0:\Z^s\to \Z$ with the property
\begin{equation}\label{9weight}
\mbox{for any integer $n\in\Z$, the set $w_0^{-1}(\,(-\infty,n]\,)$
is finite.}\end{equation}

%

\bekezdes\label{9complex} {\bf The weighted cubes.}
The space
$\Z^s\otimes \R$ has a natural cellular decomposition into cubes. The
set of zero-dimensional cubes is provided  by the lattice points
$\Z^s$. Any $l\in \Z^s$ and subset $I\subset \calv$ of
cardinality $q$  defines a $q$-dimensional cube $\square_q=(l, I)$, which has its
vertices in the lattice points $(l+\sum_{v\in I'}E_v)_{I'}$, where
$I'$ runs over all subsets of $I$.
 The set of $q$-dimensional cubes  is denoted by $\calQ_q$ ($0\leq q\leq s$).

Using $w_0$ we define
$w_q:\calQ_q\to \Z$  ($0\leq q\leq s$) by
$w_q(\square_q):=\max\{w_0(l)\,:\, \mbox{$l$ is a vertex of $\square_q$}\}$.

For each $n\in \Z$ we
define $S_n=S_n(w)\subset \R^s$ as the union of all
the cubes $\square_q$ (of any dimension) with $w(\square_q)\leq
n$. Clearly, $S_n=\emptyset$, whenever $n<m_w:=\min\{w_0\}$. For any  $q\geq 0$, set
$$\bH^q(\R^s,w):=\oplus_{n\geq m_w}\, H^q(S_n,\Z)\ \ \mbox{and}\ \
\bH^q_{red}(\R^s,w):=\oplus_{n\geq m_w}\, \widetilde{H}^q(S_n,\Z).$$
Then $\bH^q$ is $\Z$ (in fact, $2\Z$)-graded, the
$2n$-homogeneous elements $\bH^q_{2n}$ consist of  $H^q(S_n,\Z)$.
Also, $\bH^q$ is a $\Z[U]$-module; the $U$-action is given by
the restriction map $r_{n+1}:H^q(S_{n+1},\Z)\to H^q(S_n,\Z)$.
The same is true for $\bH^*_{red}$.
 Moreover, for
$q=0$, a fixed base-point $l_w\in S_{m_w}$ provides an augmentation
(splitting)
 $H^0(S_n,\Z)=
\Z\oplus \widetilde{H}^0(S_n,\Z)$ for any $n\geq m_w$, hence an augmentation of the graded
$\Z[U]$-modules (where
$\calt_{2m}^+= \Z\langle U^{-m}, U^{-m-1},\ldots\rangle$ as a $\Z$-module with its natural $U$--action)
$$\bH^0\simeq\calt^+_{2m_w}\oplus \bH^0_{red}=(\oplus_{n\geq m_w}\Z)\oplus (
\oplus_{n\geq m_w}\widetilde{H}^0(S_n,\Z))\ \ \mbox{and} \ \
\bH^*\simeq \calt^+_{2m_w}\oplus \bH^*_{red}.$$

Though
$\bH^*_{red}(\R^s,w)$ has finite $\Z$-rank in any fixed
homogeneous degree, in general,  without certain additional properties of $w_0$, it is not
finitely generated over $\Z$, in fact, not even over $\Z[U]$.

\bekezdes\label{9SSP} {\bf Restrictions.} Assume that $T\subset \R^s$ is a subspace
of $\R^s$ consisting of a union of some cubes (from $\calQ_*$). For any $q\geq 0$ define $\bH^q(T,w)$ as
$\oplus_{n\geq\min{w_0|T}} H^q(S_n\cap T,\Z)$. It has a natural graded $\Z[U]$-module
structure.  The restriction map induces a natural graded
$\Z[U]$-module homogeneous homomorphism
$$r^*:\bH^*(\R^s,w)\to \bH^*(T,w) \ \ \ \mbox{(of degree zero)}.$$
In our applications to follow, $T$ (besides the trivial $T=\R^s$ case) will be one of the following:
%
(i)  the first quadrant $(\R_{\geq o})^s$,
(ii) the rectangle $[0,c]=\{x\in \R^s\,:\, 0\leq x\leq c\}$ for some lattice point $c\geq 0$, or
(iii)  a path of composed edges in the lattice, cf. \ref{9PCl}.

\bekezdes \label{9F} {\bf The `Euler characteristic' of\, $\bH^*$.}
Fix $T$ as in  \ref{9SSP} and we will assume that each $\bH^*_{red}(T,w)$ has finite $\Z$--rank.
%
The Euler characteristic of $\bH^*(T,w)$ is defined as
$$eu(\bH^*(T,w)):=-\min\{w(l)\,:\, l\in T\cap \Z^s\} +
\sum_q(-1)^q\rank_\Z(\bH^q_{red}(T,w)).$$

\begin{lemma}\cite{NJEMS}\label{bek:LCSW} If $T=[0,c]$ for a lattice point $c\geq 0$, then
\begin{equation}\label{eq:Ecal}
 \sum_{\square_q\subset T} (-1)^{q+1}w_k(\square_q)=eu(\bH^*(T,w)).\end{equation}
 \end{lemma}

\subsection{Path lattice cohomology}\label{9PCl}\cite{Nlattice}

\bekezdes \label{bek:pathlatticecoh}
Fix $\Z^s$  as in \ref{ss:latweight} and fix also a compatible weight functions
 $\{w_q\}_q$   as in \ref{9weight}. 
 Consider also a sequence $\gamma:=\{x_i\}_{i=0}^t$  so that $x_0=0$,
$x_i\not=x_j$ for $i\not=j$, and $x_{i+1}=x_i\pm E_{v(i)}$ for
$0\leq i<t$. We write $T$ for the
union of 0-cubes marked by the points $\{x_i\}_i$ and of the
segments of type  $[x_i,x_{i+1}]$.
Then, by \ref{9SSP} we get a graded $\Z[U]$-module $\bH^*(T,w)$,
which is called the {\em
path lattice cohomology} associated with the `path' $\gamma$ and weights
$\{w_q\}_{q=0,1}$. It is denoted by $\bH^*(\gamma,w)$.
It has an augmentation with $\calt^+_{2m_\gamma}$,
where $m_\gamma:=\min_i\{w_0(x_i)\}$, and one gets the {\em reduced path lattice
cohomology} $\bH^0_{red}(\gamma,w)$ with
$$\bH^0(\gamma,w)\simeq\calt_{2m_\gamma}^+\oplus
\bH^0_{red}(\gamma,w).$$
It turns out that  $\bH^q(\gamma,w)=0$ for $q\geq 1$, hence its `Euler characteristic' can be defined as  (cf.  \ref{9F})
\begin{equation}\label{eq:euh0}
eu(\bH^*(\gamma,w)):=-m_\gamma+\rank_\Z\,(\bH^0_{red}(\gamma,w)).\end{equation}

\begin{lemma} \label{9PC2}
One has the following expression of $eu(\bH^*(\gamma,w))$ in terms of the values of $w$:
\begin{equation}\label{eq:pathweights}
eu(\bH^*(\gamma,w))=-w_0(0)+\sum_{i=0}^{t-1}\,
\max\{0, w_0(x_{i})-w_0(x_{i+1})\}.
\end{equation}
\end{lemma}


\subsection{Graded roots and their cohomologies}\label{s:grgen} \cite{NOSz,NGr}

\begin{definition}\label{def:2.1} \
  Let $\RR$ be an infinite tree with vertices $\calv$ and edges
$\cale$. We denote by $[u,v]$ the edge with
 end-vertices  $u$ and $v$.  We say that $\RR$ is a {\em graded root}
with grading $\chic:\calv\to \Z$ if

(a) $\chic(u)-\chic(v)=\pm 1$ for any $[u,v]\in \cale$;

(b) $\chic(u)>\min\{\chic(v),\chic(w)\}$ for any $[u,v],\
[u,w]\in\cale$, $v\neq w$;

(c) $\chic$ is bounded from below, $\chic^{-1}(n)$ is finite for any
$n\in\Z$, and $|\chic^{-1}(n)|=1$ if $n\gg 0$.

%

\vspace{1mm}

\noindent
An isomorphism of graded roots is a graph isomorphism, which preserves the gradings.
\end{definition}

\begin{definition}\label{9.2.6}{\bf The
  $\Z[U]$-modules associated with a graded root.}
Let us identify a graded root $(\RR,\chic)$ with its topological realization provided by vertices (0--cubes)
and segments (1--cubes). Define $w_0(v)=\chic(v)$, and $w_1([u,v])=\max\{\chic(u),\chic(v)\}$ and let $S_n$ be the
union of all cubes with weight $\leq n$. Then  we might set (as above) $\bH^*(\RR,\chi)=\oplus_{n\geq \min\chic}\
H^*(S_n,\Z)$. However, at this time $\bH^{\geq 1}(\RR,\chic)=0$; we set $\bH(\RR,\chic):=\bH^0(\RR,\chic)$.
Similarly, one defines $\bH_{red}(\RR,\chic)$ using the reduced cohomology, hence
$\bH(\RR,\chic)\simeq\calt_{2\min \chic}^+\oplus \bH_{red}(\RR,\chic)$.
\end{definition}

\bekezdes\label{bek:GRootW}{\bf The graded root associated with a weight function.}
Fix a free $\Z$-module and a weight function  $w_0$.
Consider the sequence of  topological  spaces (finite cubical  complexes) $\{S_n\}_{n\geq m_w}$
with $S_n\subset S_{n+1}$, cf. \ref{9complex}.
Let $\pi_0(S_n)=\{\calC_n^1,\ldots , \calC_n^{p_n}\}$ be the set of connected components of $S_n$.

Then  we define the  graded graph  $(\RR_w,\chic_w)$ as follows. The
vertex set  $\calv(\RR_w)$ is $\cup_{n\in \Z} \pi_0(S_n)$.
The grading $\chic_w:\calv(\RR_w)\to\Z$ is
$\chic_w(\calC_n^j)=n$, that is, $\chic_w|_{\pi_0(S_n)}=n$.
Furthermore, if  $\calC_{n}^i\subset \calC_{n+1}^j$ for some $n$, $i$ and $j$,
then we introduce an edge $[\calC_n^i,\calC_{n+1}^j]$. All the edges
 of $\RR_w$ are obtained in this way.

One verifies that
 $(\RR_w,\chic_w)$ satisfies all the required properties
of the definition of a graded root, except possibly  the last one: $|\chic_w^{-1}(n)|=1$ whenever $n\gg 0$.

The property  $|\chic_w^{-1}(n)|=1$ for $n\gg 0$ is not always satisfied.
However, the graded roots associated with connected negative definite plumbing graphs
(see below) satisfy this condition as well.

\begin{proposition}\label{th:HHzero} If $\RR$ is a graded root associated with $(T,w)$ and
 $|\chic_w^{-1}(n)|=1$ for all $n\gg 0$ then $\bH(\RR)=\bH^0(T,w)$.
\end{proposition}

\section{Combinatorial lattice cohomology} \label{ss:CombLattice}

\subsection{}
In this section we review several combinatorial statements regarding the lattice cohomology
associated with any weight function with certain combinatorial  properties. We follow \cite{AgNe1}.


\bekezdes \label{bek:comblattice}
Fix  $\Z^s$ with a fixed basis $\{E_v\}_{v\in\cV}$.
Write $E_I=\sum_{v\in I}E_v$ for $I\subset \cV$ and  $E=E_{\cV}$.
Fix also  an element
$c\in \Z^s$, $c\geq E$.
Consider the lattice points $R=R(0,c):=\{l\in\Z^s\,:\, 0\leq l\leq c\}$, and assume that
to each $l\in R$ we assign

(i)   an integer $h(l)$ such that $h(0)=0$ and $h(l+E_v)\geq h(l)$
for any $v$,

(ii)  an integer $h^\circ (l)$ such that $h^\circ (l+E_v)\leq  h^\circ (l)$
for any $v$.


Once  $h$ is fixed with (i),
a possible choice for $h^\circ $ is
 $h^{sym}$, where $h^{sym}(l)=h(c-l)$. Clearly, it depends on $c$.

\bekezdes
 We say that the $h$-function 
 satisfies the  {\it `matroid rank inequality'} if
 \begin{equation}\label{eq:matroid0}
 h(l_1)+h(l_2)\geq h(\min\{l_1,l_2\})+h(\max\{l_1,l_2\}), \ \ l_1,l_2\in R.
 \end{equation}
 This implies the {\it `stability property'},
 valid for any $\bar{l}\geq 0$ with $|\bar{l}|\not\ni E_v$, namely
 \begin{equation}\label{eq:stability}
 h(l)=h(l+E_v)\ \ \Rightarrow\  \ h(l+\bar{l})=h(l+\bar{l}+E_v).
 \end{equation}
If $\hh$ is given by a filtration (see below) then it automatically satisfies the matroid
rank inequality.

\bekezdes
We  consider the set of cubes $\{\calQ_q\}_{q\geq 0}$ of $R$ as in \ref{9complex} and
the weight function
$$w_0:\calQ_0\to\Z\ \ \mbox{by} \ \ w_0(l):=h(l)+h^\circ (l)-h^\circ (0).$$
Clearly  $w_0(0)=0$.
Furthermore, we define
$w_q:\calQ_q\to \Z$ by $ w_q(\square_q)=\max\{w_0(l)\,:\, l \
 \mbox{\,is a vertex of $\square_q$}\}$. We will use the symbol $w$ for the
 system $\{w_q\}_q$. It
defines the lattice cohomology $\bH^*(R,w)$.
Moreover, for any increasing path $\gamma$ connecting 0 and $c$
 we also have
 a path lattice cohomology $\bH^0(\gamma,w)$ as in \ref{bek:pathlatticecoh}. Accordingly,  we have the numerical
Euler characteristics  $eu(\bH^*(R,w))$, $eu(\bH^0(\gamma,w))$ and $\min_\gamma eu(\bH^0(\gamma,w))$.

\begin{lemma}\label{lem:comblat} \ \cite{AgNe1} We have   $0\leq eu(\bH^0(\gamma,w))\leq h^\circ (0)-h^\circ (c)$
for any increasing path $\gamma$  connecting  0 to $c$.
 The equality $eu(\bH^0(\gamma,w))=h^\circ (0)-h^\circ (c)$
holds if and only if for any $i$
the differences $h(x_{i+1})-h(x_i)$ and $h^\circ (x_{i})-h^\circ (x_{i+1})$ simultaneously are not nonzero.
\end{lemma}

\begin{definition}\label{def:COMPGOR}
Fix  $(h,h^\circ,R)$ as in \ref{bek:comblattice}.
 We say that the pair $h$ and $h^\circ$ satisfy the `Combinatorial Duality  Property' (CDP) if
$h(l+E_v)-h(l)$ and $h^\circ (l+E_v)-h^\circ (l)$ simultaneously cannot be nonzero
for $l,\, l+E_v\in R$. Furthermore,
 we say that $h$  satisfies  the CDP  if
 the pair $(h,h^{sym})$ satisfies  it.
\end{definition}

\begin{definition}\label{def:comblat}
We say that the pair   $(h, h^\circ) $ satisfy the

(a) {\it `path eu-coincidence'} if $eu(\bH^0(\gamma,w))=h^\circ (0)-h^\circ (c)$
for any increasing path $\gamma$.

(b)  {\it `eu-coincidence'} if $eu(\bH^*(R,w))=h^\circ (0)-h^\circ (c)$.
\end{definition}

\begin{remark}
 Example 4.3.3 of \cite{AgNe1}
 shows the following two facts.

Even if $h$ satisfies the path eu-coincidence (and $h^\circ =h^{sym}$),
in general it is not true that $\bH^0(\gamma,w)$
is independent of the choice of the increasing path.
(This statement remains valid even if we consider only the symmetric increasing paths, where a
 path $\gamma=\{x_i\}_{i=0}^t$ is symmetric if $x_{t-l}=c-x_l$ for any $l$.)

Even if $h$ satisfies both the path eu-coincidence and the eu-coincidence,
in general it is not true that $\bH^*(R,w)$ equals  any of the path lattice cohomologies
$\bH^0(\gamma,w)$ associated with a certain  increasing  path.
(E.g., in the mentioned Example 4.3.3  we have $\bH^1(R,w)\not=0$, a fact which does not hold for any
path lattice cohomology.) However, amazingly, all the Euler characteristics agree.
\end{remark}


\begin{theorem}\label{th:comblattice}
Assume that $h$ satisfies the stability property, and the pair $(h,h^\circ)$
satisfies the Combinatorial Duality  Property. Then the following facts hold.

\noindent (a) \  $(h,h^\circ)$
satisfies both the path eu- and the eu-coincidence properties:  for any increasing $\gamma$  we have
$$eu(\bH^*(\gamma,w))=eu(\bH^*(R,w))=h^\circ (0)-h^\circ (c).$$
(b)
$$\sum_{l\geq 0}\,\sum _I\, (-1)^{|I|+1} w((l,I))\, \bt^{l}=
\sum_{l\geq 0}\,\sum _I\, (-1)^{|I|+1} h(l+E_I)\,\bt^{l}.$$
\end{theorem}

\section{Analytic lattice cohomology of isolated singularities} \label{s:AnCohIs}

\subsection{Some analytic properties of isolated singularities}\label{ss:AnPrel}

\bekezdes\label{bek:PG}
 Let $(X,o)$ be an irreducible   isolated  singularity of dimension $n\geq 2$. Usually we fix a (small)
  representative
 $X$  such that  it is a contractible Stein space.
 We fix a good resolution $\phi:\tX\to X$. Set $E=\phi^{-1}(o)$ for the irreducible set, let
 $E=\cup_{v\in\cV}E_v$ be its irreducible decomposition.

 \begin{theorem} \label{th:GR} \ \cite{GrRie} {\bf (Grauert--Riemenschneider Theorem)}
 $R^i\phi_*\Omega^n_{\tX}=0$  for $i>0$.
 \end{theorem}

 If $N\subset \tX$ is a (conveniently small)
 strictly Levi pseudoconvex neighborhood  of $E$ then $H^{n-1}(N, \calO)$ is finite dimensional by
 \cite[Th. IX,B.6]{GuRo}. Furthermore, the restriction $H^{n-1}(\tX,\calO_{\tX})\to
 H^{n-1}(N, \calO_N)$ is an isomorphism \cite[Lemma 3.1]{Laufer72}. In particular,
 $H^{n-1}(\tX, \calO_{\tx})$ is finite dimensional, and we can assume that $\tX$ is a
 strictly Levi pseudoconvex neighborhood of $E$ (as $N$ above).

 \begin{theorem}\label{th:ext} \ \cite{Laufer72,YauTwo}
 $R^{n-1}\phi_*(\calO_{\tX})_o\simeq H^{n-1}(\tX,\calO_{\tX})$
 is dual as a $\C$-vector space with
 $H^0(\tX\setminus E, \Omega^{n}_{\tX})/H^0(\tX, \Omega^{n}_{\tX})$.
  \end{theorem}

We write $K_{\tX}$ for the canonical divisor, that is, $\Omega^n_{\tX}\simeq\calO_{\tX}(K_{\tX})$.

We set
 $L=H_{2n-2}(\tX, \Z)=H_{2n-2}(E,\Z)$. It is a free $\Z$-module generated by the classes of
$\{E_v\}_v$. We identify it with the group of Weil divisors supported on $E$, hence any $l\in L$ has the form
$l=\sum_v n_vE_v$ with $n_v\in \Z$. We write $l\in L_{\geq 0}$ if $l$ is effective ($n_v\geq 0$ for all $v$), and
$l\in L_{>0}$ if $l$ is non-zero effective.

\begin{theorem}\label{th:Serre} One has the {\bf Serre Duality} isomorphism:
$H^0(l, \calO_{\tX}(K_{\tX}+l))=H^{n-1}(\calO_l)^*$ for any $l\in L_{>0}$.
\end{theorem}

If $c\in L_{>0}$ with $c\gg0$ (i.e., $n_v\gg 0$ for all $v$), then $H^i(c, \calO_c)\simeq H^i(\tX, \calO_{\tX})$ for $i>0$ by
Formal Function Theorem  \cite{Hartshorne}.
Similarly, for $c\gg0$ we also have $H^0(\tX, \Omega ^n_{\tX}(c))=H^0(\tX\setminus E, \Omega^n_{\tX})$, and
\begin{equation}\label{eq:c}
H^{n-1}(\tX,\calO_{\tX})^*\simeq H^0(\tX, \Omega^n_{\tX}(c))/H^0(\tX,\Omega^n_{\tX}).
\end{equation}
More generally, for any $l> 0$ we have
\begin{equation}\label{eq:lL}
H^{n-1}(l,\calO_{l})^*\simeq H^0(\tX, \Omega^n_{\tX}(l))/H^0(\tX,\Omega^n_{\tX}).
\end{equation}
Indeed, using the exact sequence of sheaves $0\to \Omega^n_{\tX}\to \Omega^n_{\tX}(l)\to
\Omega^n_{\tX}(l)|_l\to 0$ and the Grauert--Riemenschneider vanishing we obtain that
$H^0(\tX, \Omega^n_{\tX}(l))/H^0(\tX,\Omega^n_{\tX})=H^0(l,\Omega^n_{\tX}(l))$, which is Serre dual with
$H^{n-1}(\calO_l)$.

Next, we define the Hilbert function associated with the divisorial filtration of $\calO_{X,o}$ (or of
$H^0(\tX, \calO_{\tX})$):  for any $l\in L_{\geq 0}$ set
\begin{equation}\label{eq:Hilb}
  \hh(l)=\dim \frac{H^0(\tX,\calO_{\tX})}{H^0(\tX,\calO_{\tX}(-l))}.\end{equation}
Then  $\hh$ is increasing (that is, $\hh(l_1)\geq \hh(l_2)$ whenever $l_1\geq l_2$) and  $\hh(0)=0$.

We also define another numerical invariant for any $l\geq 0$, namely
\begin{equation}\label{eq:hkor}
\hh^\circ(l)  =\dim \frac{H^0(\tX\setminus E, \Omega^n_{\tX})}{H^0(\tX, \Omega^n _{\tX}(l))}.
\end{equation}
 Then  $\hh^\circ$ is decreasing,
 $\hh^\circ (0)=h^{n-1}(\calO_{\tX})$  and $\hh^\circ (c)=0$ for $c\gg 0$.

Since $\hh$ is induced by a filtration, it satisfies the {\it matroid rank inequality}
\begin{equation}\label{eq:matroid}
\hh(l_1)+\hh(l_2)\geq  \hh(\overline{l})+\hh(l),\end{equation}
where $l=\min\{l_1, l_2\}$ and $\overline{l}=\max\{l_1,l_2\}$.
By the very same reason $l\mapsto \hh^\circ(-l)$ (hence
$l\mapsto \hh^\circ(l)$ too) satisfies the matroid rank inequality
\begin{equation}\label{eq:matroid2}
\hh^\circ (l_1)+\hh^\circ(l_2)\geq  \hh^\circ(\overline{l})+\hh^\circ(l).\end{equation}
From (\ref{eq:lL}) we obtain that \begin{equation}\label{eq:hh}
\hh^\circ(l)=h^{n-1}(\calO_{\tX})-h^{n-1}(\calO_l).\end{equation}
This shows that $l\mapsto h^{n-1}(\calO_l)$ satisfies the
{\it `opposite' matroid rank inequality}
 \begin{equation}\label{eq:matroid3}
h^{n-1}(\calO_{l_1})+h^{n-1}(\calO_{l_2})\leq  h^{n-1}(\calO_{\overline{l}})+h^{n-1}(\calO_{l}).\end{equation}
Recall that $l\mapsto h^{n-1}(\calO_l)$ is increasing and its stabilized value (for $l\gg 0$) is $h^{n-1}(\calO_{\tX})$.
\begin{proposition}\label{prop:cohcyc} {\bf (Existence of the cohomology cycle)}
Assume that $h^{n-1}(\calO_{\tX})\not=0$. Then there exists a unique minimal cycle $Z_{coh}>0$ such that
$h^{n-1}(\calO_{\tX})=h^{n-1}(\calO_{Z_{coh}})$. The cycle $Z_{coh}$ has the property that for any $l\not\geq Z_{coh}$
one has $h^{n-1}(\calO_l)<h^{n-1}(\calO_{\tX})$.
\end{proposition}
\begin{proof}
Use  (\ref{eq:matroid3}). In fact, the proof of the existence of the cohomology cycle for surface singularities
from \cite{MR} can also be adapted (which proves the opposite matroid ineqaulity as well).
\end{proof}
If $h^{n-1}(\calO_{\tX})=0$ then we define $Z_{coh}$ as the zero cycle.
\begin{corollary}\label{cor:matr}
For any $l>0$ one has $h^{n-1}(\calO_l)=h^{n-1}(\calO_{\min\{l,Z_{coh}\}})$.
\end{corollary}
\begin{proof}
Use the monotonicity of $h^{n-1}(\calO_l)$ and the opposite matroid inequality for
$l$ and $Z_{coh}$.
\end{proof}

\bekezdes It is well-known that both $h^{n-1}(\calO_{\tX})$ and $h^{n-1}(\calO_E)$ are independent of the choice of the resolution, they depend only on $(X,o)$. Moreover, the natural map
$H^{n-1}(\calO_{\tX})\to H^{n-1}(\calO_E)$ is surjective \cite[(2.14)]{Steen}.

\begin{example}\label{ex:Gorenstein}
{\bf Assume that $(X,o)$ is Gorenstein.} Then,  for any good resolution $\tX\to X$ there exists $Z_K\in L$ such that
$\Omega_{\tX}^n=\calO_{\tX}(-Z_K)$. Let us write $Z_K$ as $Z_{K,+}+Z_{K,-}$, where
$Z_{K,+}, -Z_{K,-}\in L_{\geq 0}$, and in their support there is no common $E_v$. E.g., if $(X,o)$ is rational
then $Z_{K,+}=0$. Ishii in \cite[3.7]{Ishii} proved that $Z_{K,+}\geq Z_{coh}$.

A good resolution is called `essential' if $Z_{K,-}=0$. For surface singularities essential good resolutions
exist (e.g. the minimal good resolution is such). However, in higher dimensions there are singularities
without any essential good resolutions.

Recall that in general  $\hh^\circ (l)=\dim H^0(\Omega^n _{\tX}(c))/H^0(\tX, \Omega^n _{\tX}(l))$,
valid for any $c\geq Z_{coh}$. Now, in the Gorenstein case,
if $\tX$ is an essential good resolution (i.e. $Z_K\in L_{\geq 0}$), then the previous expression for
$\hh^\circ$ transforms for $c=Z_{K}$ into
$\hh^\circ(l)=\dim H^0(\calO_{\tX})/H^0(\calO_{\tX}(-Z_K+l))$
for ant $0\leq l\leq Z_K$. Hence $\hh^\circ (l)=\hh(Z_K-l)$.
That is, $\hh^\circ$ is the symmetrized $\hh$ with respect to $Z_K\geq 0$.
\end{example}

\subsection{The analytic lattice cohomology associated with $\phi$}\label{ss:ALCc}
Let us fix some $c\geq Z_{coh}$ and we consider the rectangle $R(0,c)$. We also define the weight function on the
lattice points of $R(0,c)$ by
$$w_0(l):= \hh(l)+\hh^\circ (l)-\hh^\circ (0)=\hh(l)-h^{n-1}(\calO_l).$$
By the above discussions we obtain that $w_0$  satisfies the matroid rank inequality.

\begin{lemma}\label{lem:finiteness} Consider the case $c=\infty$, and $w_0:L_{\geq 0}\to \Z$ defined
as in \ref{ss:ALCc}. Then $w_0$ satisfies (\ref{9weight}), namely
 $w_0^{-1}(\,(-\infty,n]\,)$ is finite for any  $n\in\Z$.
\end{lemma}
\begin{proof}
Assume the opposite. Then there exists an infinite  sequence of cycles $\{l_i\}_{i\geq 1}$
 such that $\hh(l_i)\leq n$ for any $i$, and for a certain $v\in\calv$ the
 $v$--coordinates $\{l_{i,v}\}_i $ tend to infinity. Then, choose another sequence
 $\{\bar{l}_i\}$ with $\bar{l}_i\leq l_i$ so that $\bar{l}_{i,v}=l_{i,v}$ but all the other
 coordinates are bounded. For this again  $\hh(\bar{l}_i)\leq n$. Then $\{\bar{l}_i\}_i$
 admits an increasing  subsequence $\{x_j\}_j$ such that $\lim_{j\to \infty}x_{j,v}=\infty$ and
the sequence  $\{x_{j,w}\}_j  $ is constant  for any other $w\not=v$.
 Since $\hh(x_j)\leq n$, the sequence of ideal
 $H^0(\calO_{\tX}(-x_j))$ must stabilize for $j$ large. Let us
 choose some $f$ from this stabilised vector space, and let $m_v$ be its multiplicity along $E_v$.
 Then for any $j$ sufficiently large $x_{j,v}>m_v$, hence $f\not\in  H^0(\calO_{\tX}(-x_j))$,
 which is a contradiction.
\end{proof}

Furthermore,  we define
$w_q:\calQ_q\to \Z$ by $ w_q(\square_q)=\max\{w_0(l)\,:\, l \
 \mbox{\,is any vertex of $\square_q$}\}$.
 In the sequel  we write $w$ for the system $\{w_q\}_q$ if there is no confusion.

 The compatible weight functions $\{w_q\}_q$
 for any $c\geq \Z_{coh}$ (finite or infinite) define the lattice cohomology $\bH^*(R(0,c),w)$
 and a graded root $\RR(R(0,c),w)$.

\begin{lemma}\label{lem:INDEPAN}
$\bH^*(R(0,c),w)$ and  $\RR(R(0,c),w)$ are  independent of the choice of $c$
($Z_{coh}\leq c\leq \infty$).
\end{lemma}
\begin{proof}
Fix some $c\geq Z_{coh}$ and choose $E_v$ in the support of $c-Z_{coh}$.
Then for any $l\in R(0,c)$ with
$l_v=c_v$ we have $\min\{l, Z_{coh}\}=\min\{l-E_v, Z_{coh}\}$.
Therefore, by Corollary \ref{cor:matr}, $h^{n-1}(\calO_{l-E_v})=h^{n-1}(\calO_l)$, thus
$w_0(l-E_v)\leq w_0(l)$. Then for any $n\in \Z$, a strong deformation retraction  in the direction $E_v$ realizes
a homotopy equivalence between the spaces $S_n\cap R(0,c)$ and  $S_n\cap R(0,c-E_v)$.
A natural retraction
 $r:S_n\cap R(0,c)\to S_n\cap R(0,c-E_v)$ can be defined as follows (for notation see \ref{9complex}).
If $\square =(l,I)$ belongs to $ S_n\cap R(0,c-E_v)$ then $r$ on $\square$ is defined as the identity.
If $(l,I)\cap  R(0,c-E_v)=\emptyset$, then $l_v=c_v$, and  we set  $r(x)=x-E_v$. Else,
$\square =(l,I)$ satisfies $v\in I$ and $l_v=c_v-1$. Then we retract $(l,I)$ to $(l, I\setminus v)$ in the $v$--direction.
The strong deformation retract is defined similarly.
\end{proof}

\begin{corollary}\label{cor:veges}
(a) The graded root $\RR(R(0,c),w)$ satisfies $|\chic^{-1}(n)|=1$ for any $n\gg 0$.

(b) $\bH^*_{red}(R(0,c),w)$ is a finitely generated $\Z$-module (for any finite or infinite $c\geq Z_{coh}$).
\end{corollary}
\begin{proof}
For any $n\gg 0$ we have $R(0,c)=S_n$, hence $S_n$ is contractible for such $n$.
\end{proof}

In the sequel we rewrite  the $c$--independent $\bH^*(R(0,c),w)$ and
$\RR(R(0,c),w)$
as $\bH^*_{an}(\phi)$ and $\RR_{an}(\phi)$ respectively.

\subsection{The analytic lattice cohomology of $(X,o)$, independence  of $\phi$}\label{ss:anphi} \

Fix some $c\geq Z_{coh}$ and consider $R=R(0,c)$  and $\bH^*_{an}(\phi)$ as above.

\begin{theorem}\label{th:annlattinda} Assume that $h^{n-1}(\calO_E)=0$. Then
$\bH^*_{an}(\phi)$ and $\RR_{an}(\phi)$  are
 independent of the choice of the resolution  $\phi$.
\end{theorem}
\begin{proof}
By the Weak Factorization Theorem \cite{Wlod} it is  enough to show
that for a  fixed  resolution $\phi$  blowing up a smooth subvariety
of $E$
does not change $\bH^*_{an}(\phi)$. Indeed,  any two good resolutions are connected by
a sequence of such blowups and blowdowns.

So
let us fix a good resolution $\phi$ with exceptional set $E=\cup_{v\in \cV}E_v$, and blow
up a compact smooth irreducible subvariety $F$ on $E$.
Let $\pi$ be this blowup, and write $\phi':= \phi\circ \pi$. Let $E'=(\phi')^{-1}(o)$, $E'_{new}=\pi^{-1}(F)$.
We set $E_v'$ for the strict transform of $E_v$,  hence $E'=(\cup_vE'_v) \cup E'_{new}$.

Let $r\geq 2$ be  the codimension of $F$. Furthermore, let $\cF:=\{v\in\cV\,:\, F\subset E_v\}$.
Since $F\subset E$ and $F$ is irreducible, necessarily $\cF\not=\emptyset$.
Furthermore, since  $E$ is a normal crossing divisor, $|\cF|\leq r$.

%

Let $L$ and $L'$ be the corresponding free $\Z$-modules.
Associated with $\phi$,  let $\hh$  be the
 Hilbert function, $w_0$ the analytic weight and
 $S_n(\phi)=\cup\{\square\,:\, w(\square)\leq n\}$. We use similar notations  $\hh'$, $w_0'$ and
 $S_n(\phi')$ for  $\phi'$.


  We have the following natural morphisms:
 $\pi_*:L'\to L$ defined by $\pi_*(\sum x_vE'_v+x_{new}E'_{new})=\sum
x_vE_v$, and $\pi^*:L\to L' $  defined by
$\pi^*(\sum x_vE_v)=\sum x_vE'_v +(\Sigma_{v\in \cF}x_v)\cdot E'_{new}$.

 The following lemma will be used several times.

 \begin{lemma}\label{lem:pullback}
$H^0(\tX', \pi^*\calL(aE'_{new}))=H^0(\tX,\calL)$ for any $a\geq 0$ and line bundle $\calL$ on $\tX$.
\end{lemma}
\begin{proof}
The composition $H^0(\tX,\calL)\stackrel {\pi^*}{\hookrightarrow}
H^0(\tX', \pi^*\calL(aE_{new}'))\hookrightarrow
H^0(\tX'\setminus E_{new}', \pi^*\calL(aE_{new}'))\simeq
H^0(\tX\setminus F,\calL)$ is injective, and the
inclusion  $H^0(\tX,\calL)\hookrightarrow
H^0(\tX\setminus F,\calL)$
is an isomorphism since $r\geq 2$.
\end{proof}

For any $x\in R$, Lemma \ref{lem:pullback} applied for $\calL=\calO_{\tX}(-x)$ gives
\begin{equation}\label{eq:1}
 \hh'(\pi^*x+aE'_{new}) \ \left\{ \begin{array}{l}
 = \hh(x) \ \mbox{ for any $a\leq 0$} \\
 \mbox{is increasing for $a\geq 0$}.\end{array}\right.
\end{equation}
We wish a similar fact for $\hh^\circ$. First note that $K_{\tX'}=\pi^*K_{\tX}+(r-1)E'_{new}$ (see
\cite[Ex. II.8.5]{Hartshorne}. Then,
\begin{equation*}\begin{split}
h^{n-1}(\calO_{\pi^*x+aE'_{new}})&=\dim\,\frac{H^0(\tX', \Omega^n_{\tX'}(\pi^*x +a E'_{new}))}
{H^0(\tX', \Omega^n_{\tX'})}\\
\ & =\dim\,\frac{H^0(\tX', \calO_{\tX'}(\pi^*K_{\tX} +(r-1)E'_{new}+\pi^*x +a E'_{new}))}
{H^0(\tX', \calO_{\tX'}(\pi^*K_{\tX}+(r-1)E'_{new}))}.\end{split}\end{equation*}
By Lemma \ref{lem:pullback}, $H^0(\tX', \calO_{\tX'}(\pi^*K_{\tX}+(r-1)E'_{new}))=
H^0(\tX, \calO_{\tX}(K_{\tX}))=H^0(\tX, \Omega^n_{\tX})$,
while $$H^0(\tX', \calO_{\tX'}(\pi^*K_{\tX} +(r-1)E'_{new}+\pi^*x +a E'_{new}))
=H^0(\tX, \Omega^n _{\tX}(x))$$ whenever $r-1+a\geq 0$.
 Therefore,
 \begin{equation}\label{eq:2}
 h^{n-1}(\calO_{\pi^*x+aE'_{new}}) \ \left\{ \begin{array}{l}
  \mbox{is increasing  for $a\leq  1-r$}, \\
= h^{n-1}(\calO_x) \ \mbox{ for any $a\geq  1-r$}. \\
\end{array}\right.
\end{equation}
In particular, (\ref{eq:2}) applied for $a=1-r$ we obtain that
$$\mbox{if $c\geq Z_{coh}(\phi)$ \ then \
$\pi^*c-(r-1)E'_{new}\geq Z_{coh}(\phi')$ too.}$$
Indeed,
$h^{n-1}(\calO_{\pi^*c-(r-1)E'_{new}})=h^{n-1}(\calO_c)=h^{n-1}(\calO_{\tX})= h^{n-1}(\calO_{\tX'})$.

(\ref{eq:1}) and (\ref{eq:2}) combined provide
\begin{equation}\label{eq:HOM2a}
a\mapsto w_0'(\pi^*x+aE'_{new}) \ \left\{ \begin{array}{l}
\mbox{is  decreasing for $a\leq  1-r$},\\
 = w_0(x) \ \mbox{ for  $1-r\leq a\leq 0$,} \\
 \mbox{is increasing for $a\geq 0$}.\end{array}\right.
\end{equation}
Next we compare the lattice cohomology of the rectangles $R=R(0,c)$ and $R'=R(0,\pi^*c)$ associated with
$w$ and $w'$ respectively.

If $w_0'(\pi^*x+aE'_{new})\leq n$, then
$w_0(x)\leq n$ too. In particular, the projection $\pi_{\R}$ in the direction of $E'_{new}$
induces a well-defined map $\pi_{\R}:S_n(\phi')\to S_n(\phi)$.
We claim that this is a   homotopy equivalence  (with all fibers non-empty and contractible).

\bekezdes\label{bek:r1} Recall that $|\calF|\leq r$.   {\bf In the first case we assume that $|\calF|\leq r-1$. }

Our goal is to prove that $\pi_{\R}:S_n(\phi')\to S_n(\phi)$ is a homotopy equivalence.

Let us first verify that  $\pi_{\R}: S_n(\phi')\to S_n(\phi)$ is onto.

Consider a  lattice point $x\in S_n(\phi)$. Then $w_0(x)\leq n$. But then $w_0'(\pi^*x)
=w_0(x)\leq n$ too, hence $\pi^*(x)\in S_n(\phi')$ and $x=\pi_{\R}(\pi^*x)\in {\rm im}(\pi_{\R})$.

Next take  a cube $(x,I)\subset  S_n(\phi)$ ($I\subset \cV$). This means that
$w_0(x+E_{I'})\leq n$ for any $I'\subset I$. But then
\begin{equation}\label{eq:eps}
\pi^*(x+E_{I'})=\pi^*x+ E'_{I'}+\epsilon\cdot E'_{new},
\end{equation}
where $\epsilon=|I\cap \calF|$. In particular, by our assumption, $\epsilon\in \{0, \ldots, r-1\}$.
Hence
\begin{equation}\label{eq:eps2}
w_0'(\pi^*x+E'_{I'}) =w_0'(\pi^*(x+E_{I'})-\epsilon E'_{new})\stackrel {(\ref{eq:HOM2a})}{=}
w_0(x+E_{I'})\leq n.
\end{equation}
Therefore $(\pi^*x,I)\in S_n(\phi')$ and $\pi_{\R}$ projects $(\pi^*x, I)$  isomorphically  onto $(x,I)$.

Next, we show that $\pi_{\R}$ is in fact a homotopy equivalence.
In order to prove this fact it is enough to verify that if
$\square\in S_{n}(\phi)$ and $\square ^\circ$ denotes its relative interior,
then $\pi_{\R}^{-1}(\square^\circ) \cap S_{n}(\phi')$ is contractible.

Let us start again with a lattice point $x\in S_n(\phi)$. Then $\pi_{\R}^{-1}(x)\cap S_n(\phi')$ is a real  interval
(whose end-points are lattice points, considered in the real line of the $E_{new}$ coordinate).
Let us denote it by $\cI(x)$. Now, if $\square=(x,I)$, then we have to show that
all the intervals $\cI(x+E_{I'})$ associated with all the subsets $I'\subset I$ have a common
lattice point. But this is exactly what we verified above: the $E'_{new}$--coordinate of $\pi^*(x)$ is such a common
point. Therefore, $\pi_{\R}^{-1}(\square ^\circ)\cap S_n(\phi')$
has a strong deformation retraction (in the $E'_{new}$ direction)
to the contractible space $(\pi^*x, I)^\circ$.

For any $l\in L$ let  $N(l)\subset \R^s$ denote the union of all cubes which have $l$ as one of their vertices.
Let $U(l)$ be its interior. Write $U_n(l):=U(l)\cap S_n(\phi)$. If $l\in S_n(\phi)$ then $U_n(l)$ is a contractible
neighbourhood of $l$ in $S_n(\phi)$. Also,  $S_n(\phi)$ is covered by $\{U_n(l)\}_l$.
Moreover, $\pi_{\R}^{-1}(U_n(l))$ has the homotopy type of $\pi_{\R}^{-1}(l)$, hence it is contractible.
More generally, for any cube $\square$,
$$\pi_{\R}^{-1}(\cap _{\mbox{$v$ vertex of $\square$}} U_n(l)) \sim \pi_{\R}^{-1}(\square ^\circ)$$
which is contractible by the above discussion. Since all the intersections of $U_n(l)$'s are of this type,
we get that the inverse image of any intersection is contractible. Hence by \v{C}ech covering
(or Leray spectral sequence) argument, $\pi_{\R}$
induces an isomorphism
$H^*(S_n(\phi'),\Z)=H^*(S_n(\phi),\Z)$. In fact, this already shows that $\bH^*_{an}(\phi')=\bH^*_{an}(\phi)$.
In order to prove the homotopy equivalence,
one can use   quasifibration, defined in  \cite{DoldThom};
 see also \cite{DadNem}, e.g. the relevant Theorem 6.1.5.
Since
$\pi_{\R}:S_{n}(\phi')\to  S_{n}(\phi)$  is a quasifibration,
and  all the fibers are contractible, the homotopy equivalence follows.

\bekezdes  {\bf Assume now  that $|\calF|= r$. }
The proof starts very similarly. Indeed, as above, for any lattice point
 $x\in S_n(\phi)$ we have  $\pi^*(x)\in S_n(\phi')$ and $x=\pi_{\R}(\pi^*x)\in {\rm im}(\pi_{\R})$.

If we take a cube $(x,I)\subset  S_n(\phi)$ ($I\subset \cV$), then
$w_0(x+E_{I'})\leq n$ for any $I'\subset I$. We consider the identity (\ref{eq:eps}) as above.
If $|I\cap\calF|\leq r-1$ then the proof from \ref{bek:r1} works unmodified, hence
$\pi_{\R}:S_n(\phi')\to S_n(\phi)$ is a homotopy equivalence.

Assume next that  $|I\cap\calF|= r$, i.e. $ \calF\subset I$. Write $J:= I\setminus \calF$.

\bekezdes {\bf Case 1.} Let us analyse the cube $(\pi^*x,I)$  as a possible cover of $(x,I)$.

Using (\ref{eq:HOM2a}) we obtain that for any $I'\subset I$ such that $|I'\cap \calF|\leq r-1$
 we have $\pi^*x+E'_{I'}\in S_n(\phi')$.
Indeed, $$
w_0'(\pi^*x+E'_{I'}) =w_0'(\pi^*(x+E_{I'})-|I'| E'_{new})\stackrel {(\ref{eq:HOM2a})}{=}
w_0(x+E_{I'})\leq n.$$
But the vertices  $\pi^*x+E'_{I'}$, with  $|I'\cap \calF|= r$,
are  not necessarily in $S_n(\phi')$.


 However, let us  assume that $w_0'(\pi^*x+E'_I)=w_0'(\pi^*x+E'_I+E'_{new})$,
 or $w_0'(\pi^*(x+E_I)-rE'_{new})=w_0'(\pi^*(x+E_I)-(r-1)E'_{new})$.
  Then by (\ref{eq:1}) and (\ref{eq:2})
 we obtain that $h^{n-1}(\calO_{\pi^*x+E'_I})=h^{n-1}(\calO_{\pi^*x+E'_I+E'_{new}})$.  By the opposite
 matroid rank inequality of $h^{n-1}$ and (\ref{eq:1}) and (\ref{eq:2}) again we obtain that
 $w_0'(\pi^*x+E'_I-E'_{J'})=w_0'(\pi^*x+E'_I-E'_{J'}+E'_{new})$ for any $J'\subset J$.
 In particular,
 $$w_0'(\pi^*x+E'_I-E'_{J'})=w_0'(\pi^*x+E'_I-E'_{J'}+E'_{new})=w_0'(\pi^*(x+E_I-E_{J'})-(r-1)E'_{new})=
 w_0(x+E_I-E_{J'})\leq n.$$
That is, the  vertices of type $\pi^*x+E'_I-E'_{J'}$ of $(\pi^*x,I)$
are in $S_n(\phi')$. For all other vertices we already know this fact (see above, or use (\ref{eq:HOM2a})).
Hence $(\pi^*x,I)$ is in $S_n(\phi')$ and it projects via $\pi_{\R}$ bijectively to $(x,I)$.
Furthermore, $\pi_{\R}^{-1}(x,I)^\circ \cap S_n(\phi')$ admits a strong deformation retraction  to
 $(\pi^*x,I)^\circ $, hence it is contractible.

\bekezdes {\bf Case 2.} Let us analyse the second candidate,
the cube $(\pi^*x+E'_{new},I)$,  as a possible cover of $(x,I)$.

Using (\ref{eq:HOM2a}) we obtain that for any $I'\subset I$, $|I'\cap \calF|\not=0$
 we have $\pi^*x+E'_{new}+E'_{I'}\in S_n(\phi')$.
Indeed, $$
w_0'(\pi^*x+E'_{new} +E'_{I'}) =w_0'(\pi^*(x+E_{I'})+E'_{new}-|I'\cap \calF| E'_{new})\stackrel {(\ref{eq:HOM2a})}{=}
w_0(x+E_{I'})\leq n,$$
since $0\leq |I'\cap\calF|-1\leq r-1$.
But in this case, the vertices  $\pi^*x+E'_{new}+E'_{I'} $ with $I'\cap\calF=\emptyset$ are not necessarily in $S_n(\phi')$.

However, let us assume at this time that
 $w_0'(\pi^*x)=w_0'(\pi^*x+E'_{new})$. Then by (\ref{eq:1}) and (\ref{eq:2})
 we obtain that $\hh'(\pi^*x)=\hh'(\pi^*x+E'_{new})$.  By the matroid rank inequality of $\hh'$
 we get that $\hh'(\pi^*x+E'_{J'})=\hh'(\pi^*x+E'_{J'}+E'_{new})$ for any $J'\subset J$.
 This again via (\ref{eq:1}) and (\ref{eq:2})
 shows that $w_0'(\pi^*x+E'_{J'})=w_0'(\pi^*x+E'_{J'}+E'_{new})$.
 In particular,
 $$w_0'(\pi^*x+E'_{J'}+E'_{new})=w_0'(\pi^*x+E'_{J'})=w_0'(\pi^*(x+E_{J'}))=w_0(x+E_{J'})\leq n.$$
That is, the  vertices of type $\pi^*x+E'_{J'}+E'_{new}$ of $(\pi^*x+E'_{new},I)$
are in $S_n(\phi')$. For all other vertices we already know this fact
(see above).
Hence $(\pi^*x+E'_{new},I)$ is in $S_n(\phi')$ and it projects via $\pi_{\R}$ bijectively to $(x,I)$.
Furthermore, $\pi_{\R}^{-1}(x,I)^\circ \cap S_n(\phi')$ admits a deformation retraction to
 $(\pi^*x+E_{new},I)^\circ $, hence it is contractible.

\bekezdes {\bf Case 3.} If  $w_0'(\pi^*x+E'_I)=w_0'(\pi^*x+E'_I+E'_{new})$ then by Case 1 we cover
$(x,I)$ by a cube. If  $w_0'(\pi^*x)=w_0'(\pi^*x+E'_{new})$ then the same happens by Case 2.
Here in this case we assume that neither of these is satisfied, that is
 $w_0'(\pi^*x+E'_I)>w_0'(\pi^*x+E'_I+E'_{new})$ and
$w_0'(\pi^*x)<w_0'(\pi^*x+E'_{new})$.

If $w_0'(\pi^*x+E'_I)>w_0'(\pi^*x+E'_I+E'_{new})$ then
 $h^{n-1}(\calO_{\pi^*x+E'_I})< h^{n-1}(\calO_{\pi^*x +E'_I+E'_{new}})$.

If $w_0'(\pi^*x)<w_0'(\pi^*x+E'_{new})$  then  $\hh'(\pi^*x)<\hh'(\pi^*x+E'_{new})$.

These two conditions imply (use (\ref{eq:lL})):
$$\left\{ \begin{array}{l}
(a) \ \ H^0(\calO_{\tX'}(-\pi^*x-E'_{new}) \subsetneq  H^0(\calO_{\tX'}(-\pi^*x)), \ \mbox{and} \\
(b) \ \ H^0(\tX', \Omega^n_{\tX'}(\pi^* x +E'_I)) \subsetneq
H^0(\tX', \Omega^n_{\tX'}(\pi^* x +E'_I+E'_{new})).
\end{array}\right.$$
By part {\it (a)} there exists a function $f\in H^0(\tX', \calO_{\tX'})$
such that ${\rm div}_{E'}(f)\geq \pi^*x$, and in this inequality the $E'_{new}$--coordinate entries are equal.
By part {\it (b)}, there exists a global  $n$--form $\omega $ such that
${\rm div}_{E'}(\omega)\geq -\pi^*x-E'_I-E'_{new}$ and the $E'_{new}$--coordinate entries are equal.
Therefore, the form $f\omega\in H^0(\tX'\setminus E', \Omega^2_{\tX'})$ has the property that
${\rm div}_{E'}(f\omega)\geq -E'_I-E'_{new}$ with equality at the $E'_{new}$ coordinate.
In particular, again by duality (\ref{eq:lL}), we obtain that  in $\tX'$ the following
strict inequality  holds:
\begin{equation}\label{eq:ROSSZ}
h^{n-1}(\calO_{E'_I+E'_{new}})>h^{n-1}(\calO_{E'_I}) \ \ (\cV'=\cV\cup\{new\}, \ I\subset \cV).\end{equation}
But by our assumption $h^{n-1}(\calO_{E'})=0$ (a condition independent of resolution), hence
$h^{n-1}(\calO_{E'_V})=0$ for any $V\subset \calv'$.
In particular  (\ref{eq:ROSSZ}) cannot happen since  $h^{n-1}(\calO_{E'_I+E'_{new}})=h^{n-1}(\calO_{E'_I})=0$

\bekezdes This shows that case 3 cannot hold, hence either Case 1 or Case 2 hold, and in both cases
$\pi_{\R}^{-1}(x,I)^\circ \cap S_n(\phi')$ is contractible. Then the argument from \ref{bek:r1} works, which ends the proof of the theorem.
\end{proof}

\notation
In the sequel we will use for $\bH^*_{an}(\phi)$ the notation  $\bH^*_{an}(X,o)$, and
for $\RR_{an}(\phi)$ the notation  $\RR_{an}(X,o)$.
They are   called the {\it  analytic lattice cohomology } and the
{\it analytic graded root of $(X,o)$} respectively.
They are  invariants of the germ $(X,o)$.

\begin{remark}
The resolution $\tX\to X$ can be factorized through the normalization $(\overline{X},o)$ of
$(X,o)$. In particular,
$\bH^*_{an}(X,o)=\bH^*_{an}(\overline{X},o)$ and
$\RR_{an}(X,o)=\RR_{an}(\overline{X},o)$.
\end{remark}

\subsection{The `Combinatorial Duality Property' of the pair $(\hh, \hh^\circ)$}\label{ss:anCDP}

\bekezdes Next, we wish to apply Theorem \ref{th:comblattice}. Note that $\hh$ satisfies the stability property, since
it satisfies the matroid rank inequality (being induced by a filtration). Next we verify the CDP condition.

\begin{lemma}\label{lem:hsimult} Assume tat $h^{n-1}(\calO_E)=0$. Then
there exists no  $l\in L_{\geq 0}$ and  $v\in\calv$  such that  the differences
$\hh(l+E_v)-\hh(l)$ and $\hh^\circ (l)-\hh^\circ (l+E_v)$ are simultaneously strictly  positive.
\end{lemma}
\begin{proof}
If $\hh(l+E_v)>\hh(l)$ then there exists a global function $f\in H^0(\calO_{\tX})$ with ${\rm div}_Ef\geq l$, where the
$E_v$-coordinate is  $({\rm div}_Ef)_v = l_v$.
Similarly, if $\hh^\circ (l)>\hh^\circ(l+E_v)$ then there exists a
global $n$--form $\omega$ with possible poles along $E$,
 with ${\rm div}_E\omega \geq -l-E_v$, and
 $({\rm div}_E\omega )_v = -l_v-1$.
In particular, the form $f\omega$
satisfies
 ${\rm div}_E\ f\omega \geq -E_v$  and
  $({\rm div}_E f \omega )_v = -1$. This implies $H^0(\Omega_{\tX}^n(E_v))/H^0(\Omega_{\tX}^n)\not=0$, or,
  by (\ref{eq:lL}), $h^{n-1}(\calO_{E_v})\not=0$. This last fact contradicts  $h^{n-1}(\calO_E)=0$.
\end{proof}

\subsection{The Euler characteristic $eu(\bH^*_{an}(X,o)$}\label{ss:anEu}

\bekezdes Now we can apply Theorem  \ref{th:comblattice}.

Let us  consider any increasing path $\gamma$ connecting 0 and $c$
 (that is, $\gamma=\{x_i\}_{i=0}^t$, $x_{i+1}=x_i+E_{v(i)}$, $x_0=0$ and $x_t=c$, $c\geq Z_{coh}$), and let $\bH^0(\gamma,w)$
be  the  path lattice cohomology  as in \ref{bek:pathlatticecoh}.
Accordingly,  we have the numerical
Euler characteristic
 $eu(\bH^0(\gamma,w))$ as well.

\begin{theorem}\label{th:euANLAT}
$eu(\bH^*_{an}(X,o))=h^{n-1}(\calO_{\tX})$. Furthermore, for any increasing path $\gamma$ connecting 0 and $c$ (where
$c\geq Z_{coh}$)
 we also have $eu(\bH^*_{an}(\gamma,w))=h^{n-1}(\calO_{\tX})$.
\end{theorem}
This  means that $\bH^*_{an}(X,o)$ is a {\it categorification  of $h^{n-1}(\calO_{\tX})$,}
that is, it is a graded cohomology $\Z[U]$--module whose Euler characteristic is $h^{n-1}(\calO_{\tX})$.

\begin{lemma}\label{lemma:UJ}
 Assume that  $h^{n-1}(\calO_{\tX})=0$.
Then  $h^{n-1}(\calO_E)=0$ too (hence the analytic lattice cohomology and the graded root are well--defined). Furthermore, $\bH^*_{an}(X,o)=\calt^+_{0}$. In particular,
 $\bH^*_{an,red}(X,o)=0$ and the graded root
 $\RR_{an}(X,o)$ is the `bamboo' $\RR_{(0)}$: $\min\chic=0$ and $|\chic^{-1}(n)|=1$ for any $n\geq 0$.

 Conversely, if $h^{n-1}(\calO_E)=0$ (i.e. the lattice cohomology is well-defined) and
  $\bH^*_{an}(X,o)=\calt^+_{0}$ then $h^{n-1}(\calO_{\tX})=0$ too.
\end{lemma}
\begin{proof} The first statement follows from the surjectivity of $H^{n-1}(\calO_{\tX})\to H^{n-1}(\calO_E)$. Next,
we have to show that $S_n=\emptyset$ for any $n<0$ and $S_n$ is contractible for any $n\geq 0$.
Since  $Z_{coh}=0$  the rectangle $R(0,Z_{coh})$ has a single lattice point $l=0$
with  $w_0(0)=0$.  Hence $S_n\sim S_n\cap R(0, Z_{coh})=\{0\}$.
Conversely,
$\bH^*_{an}(X,o)=\calt^+_{0}$ implies  that $\min(w_0)=0$ and $eu(\bH^*_{an})=0$,
hence $h^{n-1}(\calO_{\tX})=0$  by Theorem \ref{th:euANLAT}.
\end{proof}

\subsection{Weighted cubes and the Poincar\'e series $P(\bt)$.} \label{ss:Poinc}

\ Assume that $c=\infty$, i.e.
 $R(0,c)=L_{\geq 0}$.

Recall that $\hh(l)$ was defined for any $l\in L\geq 0$.  Let us extend this definition: define
 $\hh(l)$ for any $l\in L$ by $\hh(l)=\hh(\max\{0, l\})$. This is compatible with the fact that
 $H^0(\calO_{\tX}(-l))=H^0(\calO_{\tX}(\max\{0, l\})$.

Then the multivariable Hilbert series $H(\bt)=\sum_{l\in L}\hh(l)\bt^l$ determines the
multivariable   analytic Poincar\'e series $P(\bt)=\sum_{l}\pp(l)\bt^{l }$    (cf. \cite{CDG,CHR,NJEMS}) by
\begin{equation}\label{eq:4}
P(\bt)=-H(\bt)\cdot \prod_v(1-t_v^{-1}), \ \ \mbox{or} \ \
\pp(l' )=\sum_{I\subset \{1,\ldots, s\}}\, (-1)^{|I|+1}\hh(l'+E_I).
\end{equation}
Then one verifies (using  $\hh(l)=\hh(\max\{0, l\})$) that $P(\bt)$ is supported on $L_{\geq 0}$.
Furthermore,
 Theorem  \ref{th:comblattice} and (\ref{eq:4}) combined show that   the analytic Poincar\'e series
associated with the divisorial filtration of the local ring $\cO_{X,o}$
has the following interpretation
in terms of  the (analytic) weighted cubes $\square=(l,I)$:
$$P(\bt)=\sum_{l\geq 0}\,\sum _I\, (-1)^{|I|+1} w_{an}((l,I))\, \bt^{l}$$
whenever $h^{n-1}(\calO_E)=0$.

\begin{remark}\label{rem:san}
Let $u$ be   a vertex of $\RR_{an}(X,o)$ of valency one.
This means that it is a local minimum of $\chic$ with respect to the natural partial ordering
 given by the edges and $\chic$.
  Set $n:=\chic(u)$ and let $\calC=\calC^i_n$ be the connected component of $S_n$
  which represents $u$, cf. \ref{bek:GRootW}.
  Let $l_m\in L$ be the maximal element of $\calC$ with respect to $\leq $.
  (In fact, using the matroid rank inequality of $w_0$, cf. \ref{ss:ALCc}, one shows that $l_m$ is unique.)
  Then $w_0(l_m+E_v)>w_0(l_m)$ for any $v\in\calv$. By CDP we also obtain
  $\hh(l_m+E_v)>\hh(l_m)$ for any $v\in\calv$. In particular, there exists a function $f:(X,o)\to (\C,0)$ such that
  the restriction to $E$ of the divisor of $f\circ \phi$  is $l_m$. In other word,
  $l_m$ is in the analytic semigroup $\calS_{an}$ associated with $\phi$.
Hence, local minimums of $\RR_{an}(X,o)$  represent elements of $\calS_{an}$.
In this way we cannot represent all the elements of $\calS_{an}$, since by Lemma
\ref{lem:INDEPAN} $\calC$ must also contain a lattice point in $R(0,Z_{coh})$ too.

\end{remark}

\subsection{Analytic Reduction Theorem}\label{ss:anRT}

\bekezdes
Our next goal is to prove a `Reduction Theorem'.
Via such a result, the rectangle $R=R(0,c)$ can be replaced  by another rectangle sitting in a
lattice of smaller rank. The procedure starts with identification of a set of `bad' vertices.
More precisely,   we decompose $\calv$ as a disjoint union
$\overline{\calv}\sqcup \calv^*$, where the vertices $\overline {\calv}$ are the `essential' ones, the ones which dominate the others, and the coordinates $\calv^*$ are those which `can be eliminated'.
The goal is to replace the rectangle $R$ (or $\Z^s_{\geq 0}$)  with a rectangle of $\Z^{\bar s}$, with
$\bar{s}=|\overline{\calv}|$.

In the  topological case of surface singularities  the possible choice of $\overline{\calv}$ was dictated by combinatorial
properties of the Riemann--Roch expression $\chi$ (the topological weight function),
 with a special focus on the topological characterization of rational germs \cite{NOSz,LN1}.
In the analytical case of surface singularities we used certain
  analytic properties of 2--forms \cite{AgNe1}.
  The present  high--dimensional case is a direct  generalization  of this.

\bekezdes Let $(X,o)$ be an isolated singularity of dimension $n\geq 2$, and
we fix a good resolution $\phi$ as above.

\begin{definition}\label{def:DOMAN}
We say that the subset  $\overline{\calv}$ of $\calv$  is an B$_{an}$--set
if it satisfy the following
property: if 
some differential form
$\omega\in H^0(\tX\setminus E, \Omega_{\tX}^n) $ satisfies
$({\rm div}_E\omega) |_{\overline{\calv}}\geq -E_{\overline{\calv}}$ \
then necessarily
$\omega\in H^0(\tX, \Omega_{\tX}^n)$.
By (\ref{eq:lL}) this is equivalent with the vanishing  $h^1(\calO_Z)=0$ for any
$Z=E_{\overline{\calv}}+l^*$, where $l^*\geq 0$ and it is supported on $\calv^*$.
\end{definition}

\bekezdes Associated with a disjoint  decomposition  $\calv=\overline{\calv}\sqcup \calv^*$,  we write
any  $l\in L$
as $\overline{l}+l^*$, or $(\overline{l}, l^*)$,
where $\overline {l}$ and  $l^*$ are  supported on $\overline{\calv}$ and
 $\calv^*$ respectively. We also write  $\overline{R}$ for the rectangle $R(0, \overline{c})$, the $\overline{\calv}$-projection of $R(0,c)$ with $c\geq Z_{coh}$.

For any $\overline {l}\in \overline {R}$ define  the weight function
$$\overline{w}_0(\overline{l})=\hh(\overline{l})+\hh^\circ (\overline{l}+c^*)-h^{n-1}(\calO_{\tX})
=\hh(\overline{l})-h^1(\calO_{\overline{l}+c^*}).$$
Consider all the cubes of $\overline{R}$ and the weight function
$\overline{w}_q:\calQ_q(\overline{R})\to \Z$ by $ \overline{w}_q(\square_q)=\max\{w_0(\overline{l})\,:\, \overline{l} \
 \mbox{\,is any vertex of $\square_q$}\}$.

\begin{theorem}\label{th:REDAN} {\bf Reduction theorem for the analytic lattice cohomology.}
If  $\overline{\calv}$ is an B$_{an}$-set
then
$$\bH^*_{an}(R,w)=\bH^*_{an}(\overline{R}, \overline{w}).$$
\end{theorem}
\begin{proof}
For any $\cali\subset \calv$ write $c_\cali$ for the $\cali$-projection of $c=Z_{coh}$.
We proceed by induction, the proof will be given in $|\calv^*|$ steps.
For any $\overline{\calv}\subset \cali\subset \calv$ we create the inductive setup.
We write $\cali^*=\calv\setminus \cali$, and according to the disjoint union
$\cali\sqcup \cali^*=\calv$ we  consider the coordinate decomposition
$l=(l_\cali,l_{\cali^*})$. We also set $ R_\cali=R(0, c_\cali)$ and the weight function
\begin{equation}\label{eq:WBAR}
w_\cali(l_\cali)=\hh(l_\cali)+\hh^\circ(l_\cali+c_{\cali^*})-h^{n-1}(\calO_{\tX}).
\end{equation}
Then  for $\overline{\calv}\subset \cali\subset \cJ\subset \calv$, $\cJ=\cali\cup \{v_0\}$
($v_0\not\in\cali$),
we wish to prove that $\bH^*_{an}(R_\cali, w_\cali)=\bH^*_{an}(R_{\cJ}, w_{\cJ})$.
For this consider the projection $\pi_{\R}:R_{\cJ}\to R_{\cali}$.

For any fixed $y\in R_\cali$ consider the fiber $\{y+tE_{v_0}\}_{0\leq t\leq c_{v_0},\ t\in \Z}$.

Note that $t\mapsto \hh(y+tE_{v_0})$  is increasing. Let $t_0=t_0(y)$
be the smallest value $t$ for which
$\hh(y+tE_{v_0})< \hh(y+(t+1)E_{v_0})$.
If $t\mapsto \hh(y+tE_{v_0})$ is constant then we take $t_0=c_{v_0}$.
If $t_0<c_{v_0}$, then $t_0$ is characterized by the existence of a function
\begin{equation}\label{eq:1RED}
f\in H^0(\calO_{\tX}) \ \ \mbox{with} \ \
({\rm div}_Ef)|_\cali\geq y, \ \ \  ({\rm div}_Ef)_{v_0}=t_0.
\end{equation}
Symmetrically,  $t\mapsto \hh^{\circ} (y+c_{\cJ^*}+ tE_{v_0})$  is decreasing. Let $t_0^\circ=t_0^\circ (y)$
be the smallest value $t$ for which
$\hh^{\circ} (y+c_{\cJ^*}+tE_{v_0})=\hh^{\circ} (y+c_{\cJ^*}+(t+1)E_{v_0})$. The value
$t_0^\circ$  is characterized by the existence of a form
\begin{equation}\label{eq:2RED}
\omega \in H^0(\tX\setminus E,\Omega_{\tX}^n) \ \ \mbox{with} \ \
({\rm div}_E\omega) |_\cali\geq - y, \ \ \  ({\rm div}_E\omega )_{v_0}=-t^{\circ}_0.
\end{equation}
This shows that there exists a form $f\omega\in H^0(\tX\setminus E, \Omega_{\tX}^n)$ such that
$({\rm div}_Ef\omega )|_\cali\geq 0$ and $({\rm div}_Ef\omega )_{v_0}=t_0-t^{\circ}_0$.
By the B$_{an}$ property we necessarily must have $t_0-t^{\circ}_0\geq 0$.
Therefore, the weight $t\mapsto w_{\cJ}(y+tE_{v_0})=\hh(y+tE_{v_0})+\hh^{\circ } (y+tE_{v_0}+c_{\cJ^*})-h^{n-1}(\calO_{\tX})$
is decreasing for $t\leq t_0^\circ$, is increasing for $t\geq t_0$. Moreover, for $t_0^\circ \leq t\leq t_0$
it  takes the
constant value $\hh(y)+\hh^{\circ } (y+c_{v_0}E_{v_0}+c_{\cJ^*})-h^{n-1}(\calO_{\tX})=w_{\cali}(y)$.

Next we fix $y\in R_\cali$ and some $I\subset \cali$ (hence a cube $(y,I)$ in $R_{\cali}$).
We wish to compare the intervals
$[t_0^\circ (y+E_{I'}), t_0 (y+E_{I'})]$ for all subsets $I'\subset I$. We claim that they have at least one
common element (in fact, it turns out that $t_0(y)$ works).

Note that $\hh(y+tE_{v_0})=\hh(y+(t+1)E_{v_0})$ implies $\hh(y+tE_{v_0}+E_{I'})=\hh(y+(t+1)E_{v_0}+E_{I'})$
for any $I'$,
hence $t_0(y)\leq t_0(y+E_{I'})$. In particular, we need to prove that $t_0(y)\geq t_0^\circ (y+E_{I'})$.
Similarly as above, the value $t_0^\circ(y+E_{I'})$  is characterized by the existence of a form
\begin{equation*}
\omega_{I'} \in H^0(\tX\setminus E, \Omega_{\tX}^n) \ \ \mbox{with} \ \
({\rm div}_E\omega_{I'}) |_\cali\geq - y-E_{I'}, \ \ \  ({\rm div}_E\omega_{I'} )_{v_0}=-t^{\circ}_0(y+E_{I'}).
\end{equation*}
Hence the  from $f\omega_{I'}\in H^0(\tX\setminus E, \Omega_{\tX}^n)$ satisfies
${\rm div}_Ef\omega_{I'} |_\cali\geq -E_{I'}$ and $({\rm div}_Ef\omega )_{v_0}=t_0(y)-t^{\circ}_0(y+E_{I'})$.
By the B$_{an}$ property we must have $t_0(y)-t^{\circ}_0(y+E_{I'})\geq 0$.

Set $S_{\cJ,n}$ and $S_{\cali,n}$ for the lattice spaces defined by $w_\cJ$ and $w_\cali$. If
$y+tE_{v_0}\in S_{\cJ,n}$ then $w_\cJ(y+tE_{v_0})\leq n$, hence  by the above discussion $w_\cali(y)\leq n$ too.
In particular, the projection $\pi_{\R}:R_\cJ\to R_\cali$ induces a map $S_{\cJ,n}\to S_{\cali,n}$.
We claim  that it is a homotopy equivalence. The argument is similar to the proof from
\ref{th:annlattinda} via the above preparations.
\end{proof}

\begin{corollary}\label{cor:AR} If $(X,o)$ admits a resolution $\phi$ with a B$_{an}$--set of cardinality $\overline{s}$,
then  $\bH^{\geq \overline{s}}_{an}(X,o)=0$.
\end{corollary}
\begin{example}\label{ex:Gorenstein2}
Assume that $(X,o)$ is Gorenstein, cf. Example \ref{ex:Gorenstein}.
Let $\overline{\calv}$ be a B$_{an}$--set  and assume that $Z_K|_{\overline{\calv}}\geq 0$.
Since $Z_{coh}\leq Z_{K,+}$ (cf. \cite{Ishii}), we can take $c=Z_{K,+}$. Then, for any
$\overline{l}\in R(0, Z_{K,+})$,   the
$\hh^\circ$--contribution $ \hh^\circ (\overline{l}+Z_{K,+}^*)$
in $\overline {w}_0$ is
$$
\dim\, \frac{H^0(\tx, \Omega^n_{\tX}(Z_{K,+}))}
{H^0(\tX,\Omega^n _{\tX}(\overline{l}+Z_{K,+}^*))}=
\dim\, \frac{H^0(\tx, \calO_{\tX}(-Z_K+Z_{K,+}))}
{H^0(\tX,\calO_{\tX}(-Z_K+\overline{l}+Z_{K,+}^*))}=
\dim\, \frac{H^0(\tX, \calO_{\tX}(-Z_{K,-}))}
{H^0(\tX,\calO_{\tX}(-Z_{K,-}+\overline{l}-\overline{Z_{K}}))}.
$$
Note that $-Z_{K,-}\geq 0$. On the other hand, for any $a\geq 0$ we have
\begin{equation}\label{eq:van1}
H^0(\tX,\calO_{\tX}(a))=H^0(\tX,\calO_{\tX}).
\end{equation}
Indeed, using the exact sequence $0\to \calO_{\tX}\to \calO_{\tX}(a)\to\calO_a(a)\to 0$ it is enough
to prove that $H^0(\calO_a(a))=0$. But this follows by Serre duality and Grauert--Riemenschneider
vanishing. Next, note that for $a,b\geq 0$,  both supported on $E$ but without common $E_v$--term
 in their supports,  one has
$$H^0(\tX,\calO_{\tX}(a-b))=H^0(\calO_{\tX}(-b))\cap H^0(\calO_{\tX}(a))
\stackrel{(\ref{eq:van1})}{=}
H^0(\calO_{\tX}(-b))\cap H^0(\calO_{\tX})=
H^0(\calO_{\tX}(-b)).$$
In particular,
$$ \hh^\circ (\overline{l}+Z_{K,+}^*)=\dim\, \frac{H^0(\tX, \calO_{\tX})}
{H^0(\tX,\calO_{\tX}(\overline{l}-\overline{Z_{K}}))}=
\hh(\overline{Z_K}-\overline{l}).$$
Therefore, the weight function $\overline{w}_0$ on $R(0, \overline{Z_K})$ is
\begin{equation}\label{eq:sym}
\overline{w}_0(\overline{l})= \hh(\overline{l})+\hh(\overline{Z_K}-\overline{l})- h^{n-1}(\calO_{\tX}).\end{equation}
That is, $\overline{w}_0$ is obtained by the symmetrization of the restriction of $\hh$ to
$R(0, \overline{Z_K})$.
\end{example}

\bekezdes Under the assumption $h^{n-1}(\calO_E)=0$, if $(X,o)$ is Gorenstein, then
B$_{an}$--sets $\overline {\calv}$
with $Z_K|_{\overline{\calv}}\geq 0$ exist
 for any  good resolution $\tX\to X$.  Indeed, we have the following fact.

\begin{lemma}\label{lem:exists}
If $(X,o)$ is Gorenstein and
$h^{n-1}(\calO_E)=0$ then the support of $ Z_{K,+} $ is a B$_{an}$--set.
\end{lemma}
\begin{proof} Denote the support of $Z_{K,+}$ by $I$.
Assume that $h^{n-1}(\calO_{E_I+l^*})\not=0$, where $l^*\geq 0$ and it is supported on $\calv\setminus I$. But by Corollary \ref{cor:matr} we also have
$h^{n-1}(\calO_{E_I+l^*})=h^{n-1}(\calO_{\min\{ E_I+l^*, Z_{coh}\}})$. Since
$Z_{coh}\leq Z_{K,+} $ \cite{Ishii}, we get  that $h^{n-1}(\calO_{E_I})\not=0$.
But this contradicts   $h^{n-1}(\calO_E)=0$.
\end{proof}

\begin{remark}\label{rem:exists}
Assume that $n=2$, $h^{n-1}(\calO_E)=0$, and $(X,o)$ is normal but  not  necessarily Gorenstein.
Then the rational cycle $Z_K\in L\otimes \Q$ can be defined, it is the (anti)canonical cycle,
numerically equivalent with $K_{\tX}$ (see e.g. \cite{AgNe1}).
Then again $Z_{coh}\leq\lfloor Z_{K,+} \rfloor$ (see e.g \cite{AgNe1}), hence the very same proof
gives the following: if
$h^{n-1}(\calO_E)=0$ then the support $\lfloor Z_{K,+} \rfloor$ is a B$_{an}$--set.

 If we choose  $\overline{\calv}$ as the support of $\lfloor Z_{K,+} \rfloor$,
and we also take $c=\lfloor Z_{K,+} \rfloor$, then in (\ref{eq:WBAR}) $c^*=0$ and
$$\overline{w}_0(\overline{l})=\hh(\overline{l})+\hh^\circ (\overline{l})-\hh^\circ (0)$$
for any $\overline{l}\in R(0,c)$.
\end{remark}
\section{$h^{n-1}(\calO_E)$ and the cohomology of the link}

\subsection{$h^{n-1}(\calO_{\tX})$ and the geometric genus}\label{ss:CM} \

Let $(X,o)$ be an isolated complex singularity  and let $\phi:\tX\to X$ be a good resolution as above.
Let $\overline{\calO_{X,o}}$ be the normalization of $\calO_{X,o}$ and let $\delta(X,o)$ be
the delta invariant $\dim \, (\overline{\calO_{X,o}}/\calO_{X,o})$ of $(X,o)$.
Note that $\overline{\calO_{X,o}}$ can also be identified with $(\phi_*\calO_{\tX})_o$.
As usual, we define the {\it geometric genus} by
$$(-1)^{n+1}p_g(X,o):= \delta(X,o)+\sum_{i\geq 1}(-1)^i\, h^i(\tX,\calO_{\tX}).$$
The following facts are well--known \cite{YauTwo,Karras}

\vspace{1mm}

(i) \ $h^i(\tX, \calO_{\tX})=0$ for $i\geq n$;

(ii) \ $h^{n-1}(\tX, \calO_{\tX})=\dim\,
( H^0(\tX\setminus  E, \Omega^n_{\tX})/ H^0(\tX, \Omega^n_{\tX}))$, cf. (\ref{eq:lL});

(iii) \  $h^i(\tX, \calO_{\tX})=\dim\, H^{i+1}_{\{o\}} (X,\calO_{X})$ for $1\leq i\leq n-2$;

(iv) \ If $(X,o)$ is Cohen--Macaulay then $(X,o)$ is normal  and $h^i(\tX,\calO_{\tX})=0$ for
$1\leq i\leq n-2$. In particular, $p_g(X,o)=h^{n-1}(\tX, \calO_{\tX})$.
(For this use the characterization of the Cohen--Macaulay property in terms of the local cohoomlogy,
cf. part (iii).)

\bekezdes Recall that $(X,o)$ is called {\it rational} if $(X,o)$ is normal and
$h^i(\tX,\calO_{\tX})=0$ for
$i\geq 1$. If $(X,o)$ is Cohen--Macaulay and $h^{n-1}(\calO_E)=0$ then $(X,o)$ is rational if and only if
$\bH^*_{an}(X,o)=\calt ^+_0$, cf. Lemma \ref{lemma:UJ}.

\subsection{$h^{n-1}(\calO_{E})$ and the link of $(X,o)$} \

Recall that all the (local) cohomology groups $H^*_{\{x\}}(X)$, $H^*_E(\tX)$ and $H^*(E)$ admit mixed Hodge structures. The Hodge filtration will be denoted by $F^{\cdot}$.
By \cite[Corollary 1.2]{Steen} we have the following short exact sequence
\begin{equation}\label{H1}
0\to {\rm Gr}^n_FH^{n+1}_{\{x\}}(X)\to {\rm Gr}^n_FH^{n+1}_E(\tX)\to {\rm Gr}^n_FH^{n+1}(E,\C)\to 0.
\end{equation}
Let $M$ denote the link of $(X,o)$, an oriented $(2n-1)$--dimensional compact manifold.
By \cite[Corollary 1.15]{Steen} we have an isomorphism $H^{k+1}_{\{x\}}(X)=H^k(M,\C)$ for any $1\leq k\leq 2n-2$, which is compatible with  the mixed Hodge structures. Furthermore, by \cite[page 516]{Steen}
we also have an isomorphism $H^{n-1}(E, \calO_E)={\rm Gr}^0_FH^{n-1}(E,\C)$, which is dual to
${\rm Gr}^{n}_FH^{n+1}_E(\tX)$.
Therefore, we have an exact sequence
\begin{equation}\label{H2}
0\to {\rm Gr}^n_FH^{n}(M)\to H^{n-1}(\calO_E)^*\to {\rm Gr}^n_FH^{n+1}(E,\C)\to 0.
\end{equation}
Hence, the vanishing $h^{n-1}(\calO_E)=0$ implies the vanishing ${\dim\, \rm Gr}^n_FH^{n}(M)=0$.

\subsection{$h^{n-1}(\calO_E)$ and a smoothing of $(X,o)$}

Assume that $({\mathcal X},o)$ is a complex space of dimension $n+1$ with at most an isolated singularity at $o$. Let $f:({\mathcal X},o)\to (\C,0)$ be a flat map such that $(f^{-1}(0),o)\simeq (X,o)$.
We also assume that both ${\mathcal X}$ and $X$ are contractible Stein spaces and
$f$ induces a differentiable locally trivial fibration over a small punctured disc.
Let $X_\infty$ be the
nearby (Milnor) fiber $f^{-1}(t)$ for $t\not=0$ sufficiently small. It is an $n$--dimensional complex Stein manifold.
 Note that $\partial X_\infty\simeq M$.

 Then $H^*(X_\infty)$ carries a mixed Hodge structure. Then by \cite[Proposition 2.13]{Steen} $\dim \,
 {\rm Gr}_F^nH^n(X_\infty)=p_g(X,o)$. Furthermore, by \cite[Proposition 1.15]{Steen}
\begin{equation}\label{H3}
h^{n-1}(\calO_E)=p_g(X,o)-\dim\, {\rm Gr}^0_FH^n(X_\infty)+\dim\, {\rm Gr}^0_FH^{n-1}(X_\infty).
\end{equation}
If $\widetilde{{\mathcal X}}\to {\mathcal X}$ is an embedded good resolution of the pair
$({\mathcal X}, X)$ with exceptional space ${\mathcal E}\subset \widetilde{{\mathcal X}}$ then by \cite[page 526]{Steen} we also have
${\rm Gr}^0_F H^k(X_\infty)\simeq H^k(\calO_{{\mathcal E}})$.
Hence we have
\begin{equation}\label{H4}
h^{n-1}(\calO_E)=\dim \,
 {\rm Gr}_F^nH^n(X_\infty)-\dim\, {\rm Gr}^0_FH^n(X_\infty)+h^{n-1}(\calO_{{\mathcal E}}).
\end{equation}
E.g., if $(X,o)$ is Cohen--Macaulay (e.g. if it is complete intersection), then $h^{n-1}(\calO_{{\mathcal E}})=0$ (cf. \cite[Corollary 2.16]{Steen}), hence  we have a description of $h^{n-1}(\calO_E)$ in terms of the Hodge filtration of
$H^n(X_\infty)$.

\subsection{$h^{n-1}(\calO_E)$ and $h^{n-1}(\calO_{\tX})$ for hypersurface singularities}\

Assume that $(X,o)$ is an isolated hypersurface singularity with Milnor number $\mu$. Then the (Hodge)
spectrum  consists of $\mu$ rational numbers in the interval $(0, n+1)$. Their position is symmetric
with respect to $(n+1)/2$.
The number of spectral numbers in $(0,1]$ (or, symmetrically, in $[n,n+1)$) is $p_g(X,o)=h^{n-1}(\calO_{\tX})$.
More generally, the dimension of ${\rm Gr}^k_F H^n(X_\infty)$ is the number of spectral numbers in the interval $[k,k+1)$. Hence,  $\dim\,{\rm Gr}^0_F H^n(X_\infty)$ is the number of spectral numbers in the interval $(0,1)$. Since in this case $h^{n-1}(\calO_{{\mathcal E}})=0$,
(\ref{H4}) gives that
\begin{equation}\label{eq:spec}
h^{n-1}(\calO_E)=\mbox{\{number of spectral numbers $=1$\}}=\dim\,  {\rm Gr}^n_FH^n(M).
\end{equation}
In particular, if $M$ is a rational homology sphere  then $h^{n-1}(\calO_E)=0$.

\section{Examples}\label{s:Example}

\subsection{Isolated weighted homogeneous hypersurface singularities}\label{ss:IHS}

\bekezdes Assume that $(X,o)\subset (\C^{n+1},0)$ is defined by  a weighted homogeneous
polynomial $f(z_0,\ldots, z_n)$ of weights $(w_0,\ldots, w_n)$ and degree $d$.
Here $w_i\in\Z_{>0}$, ${\rm gcd}(w_0,\ldots, w_n)=1$ and all the nontrivial monomials of $f$
have the form $c_kz^k=c_kz_0^{k_0}\cdots z_n^{k_n}$, where $c_k\in \C^*$ and
$\sum_i k_iw_i=d$.
We assume that $(X,o)$ has an isolated singularity.

A partial resolution of $X$ can be obtained by a weighted blow up of $o\in X$. This creates an exceptional set  $X_\infty$. Then we can continue the resolution procedure and we construct  a good
resolution $\tX\to X$. Let $\calv$ be the index set of irreducible exceptional sets, as above.
 We denote the strict transform of $X_\infty$ by $E_{\infty}$.
 It is irreducible (see below),  let it be indexed by $v_\infty\in\calv$.
  We wish to show that $\{v_\infty\}\subset \calv$ is a B$_{an}$--set.

  Let us provide more details about the weighted blow up. Consider the weighted projective
  $n$--space $\bP^{n}_{w}$ of weights $(w_0,\ldots, w_n)$ and the hypersurface
  $X_\infty\subset \bP^{n}_{w}$ given
  by  the equation $f(z_0,\ldots, z_n)=0$.
  Since $X$ has an isolated singularity, $X_\infty$ is necessarily irreducible.
Next,  take the incidence variety
  $$I:=\{(v,[u])\in X\times \bP^{n}_{w},\ \ v\in X\setminus \{0\}, \ [v]_w= [u]_w \, \},$$
  where $[u]_w$ denotes the class of $u$ in $\bP^n_w$.
Then its closure with the restriction of the first projection is the needed weighted blow up
with exceptional set $X_\infty$.

The space  $X_\infty$ is a $V$--manifold \cite{SteenW}. This means that its singularities locally
are quotients of type $\C^n/G$, where $G$ is a finite small subgroup of ${\rm GL}(n, \C)$.
Since quotient singularities are Cohen--Macaulay \cite{HR} and rational \cite{Vi}, we have to resolve
in continuation only `mild' singularities. In particular,  we can expect that the only `significant'
irreducible exceptional set is $E_\infty$.

\begin{lemma}\label{lem:bad} If \ $h^{n-1}(\calO_E)=0$ then
$\{v_\infty\}$ is a $B_{an}$ subset of $\calv$.
\end{lemma}
\begin{proof}
We need to show that $h^{n-1}(\calO_{E_\infty+l^*})=0$ for any $l^*\geq 0 $ supported on
$\calv^*=\calv\setminus \{v_\infty\}$. By (\ref{eq:lL}) this means that
there  exists no  differential form $\omega\in H^0(\tX\setminus E, \Omega^n_{\tX})$
such that its pole has order one along $E_\infty$.

The point is that we know all the candidate differential forms.
Indeed, since $(X,o)$ is Gorenstein, it admits a Gorenstein form $\omega_G$ (unique up to
a non--zero constant). In fact, it is
the restriction of $dx_0\wedge \cdots \widehat{dx_i}\cdots \wedge dx_n/ (\partial f/\partial x_i)$ to
$X\setminus \{o\}$. Then consider all the monomials of type $z^k$ with $\sum_i(k_i+1)w_i\leq d$.
These correspond to the lattice points $k+(1,\ldots, 1)$ with strictly positive coordinates not above the Newton boundary.
Then the classes of the differential forms $z^k\omega_G$ form a basis of
$H^0(\tX\setminus E,\Omega^n_{\tX})/H^0(\tX, \Omega^n_{\tX})$ \cite{MeTe}.
Note also that the divisorial filtration associated with $E_\infty$ agrees with the combinatorial filtration associated with the Newton diagram \cite{Lem,BaldurFiltr}.
Hence, the pole order of any linear combination
$\sum_kc_kz^k$ $(c_k\in\C$) is $\max\{\mbox{pole order } \ z^k\ :\ c_k\not=0\}$.
Therefore,
 the pole order of any linear combination
$\sum_kc_kz^k\omega_G$ $(c_k\in\C$) is $\max\{\mbox{pole order } \ z^k\omega_k\ :\ c_k\not=0\}$.
Moreover, since  the pole order of $\omega_G$ along $E_\infty$ is
$d+1-\sum_i w_i$ (cf. \cite{MeTe}), the pole order of $z^k\omega_G$ is
$d+1-\sum_i (k_i+1)w_i$. Hence, this is 1 exactly when
$\sum_i(k_i+1)w_i=d$, i.e., if the corresponding lattice point is on the  Newton boundary.
But such lattice points from the Newton boundary
 produce spectral numbers 1 \cite{StOslo,SaitoNN}. Since there exists no spectral number equal to 1 (cf.
\ref{eq:spec})), there exists no such  differential form either.
\end{proof}

\bekezdes In particular, we can apply the Analytic Reduction Theorem for this vertex.
In the next paragraphs we show that the weight function $\overline{w}_0$
of the reduction is determined by the set of
spectral numbers from the interval $(0,1)$ (recall that their number is  exactly $p_g$).

By the Reduction Theorem we have to determine $\hh(\overline{l})$ for $\overline{l}\in
 R(0, \overline{Z_K}\,)=
\Z\cap [0, d+1-\sum_i w_i]$.

Note that both $\calO_{\C^{n+1},0}$ and $\calO_{X,o}$ are graded local algebras, graded by ${\deg}(z^k)=\sum_ik_iw_i$. For $\hh(\overline{l})$ we need  to know $(\calO_{X,o})_{{\rm deg}<d+1-\sum_iw_i}$. Since $\deg(f)=d$, the homogeneous components of
$\calO_{\C^{n+1},0}$ and $\calO_{X,o}$ in degrees ${\rm deg}<d+1-\sum_iw_i$ are the same.
They are determined by the lattice points $k\in\Z_{\geq 0}^{n+1}$ with $\sum_ik_iw_i\leq d-\sum_iw_i$.
These points correspond bijectively to the lattice points $k+(1,\ldots ,1)\in \Z_{>0}^{n+1}$
 with $\sum_i(k_i+1)w_i\leq d$. These  lattice points are those which are not above the Newton boundary. But there are  no lattice points on the boundary, hence these lattice points are all under the boundary, and they
 provide the spectral numbers. Each such lattice point contributed in the spectrum by
 $\alpha=\sum_i (k_i+1)w_i/d$.

 Let $P(t)=\sum _{\ell\geq 0}p(\ell)t^{\ell}$ be the Poincar\'e series of the graded algebra $\calO_{X,o}$, and $P(t)_{<d+1-\sum_iw_i}$ be the Poincar\'e polynomial  counting the dimensions of the homogeneous components of degree $<d+1-\sum_iw_i$. Then $\hh(\overline{l})=\sum_{\ell<\overline{l}}p(\ell)$.

 If $\{\alpha_1,\ldots, \alpha_\mu\}$ are the spectral numbers of $(X,o)$ then we write
 ${\rm Spec}(t)=\sum_j t^{\alpha_j}$. Let ${\rm Spec}_{(0,1)}(t)$ be
 $\sum_{\alpha_j<1}t^{\alpha_j}$. Then it is known that \cite[5.11]{StOslo}
 $${\rm Spec}(t)=\prod _i \ \frac{t^{w_i/d}-t}{1-t^{w_i/d}}.$$
 From above we also have
 $${\rm Spec}_{(0,1)}(t)= \sum_{k_i\geq 0, \ \sum_i (k_i+1)w_i<d}\ t^{\sum_i(k_i+1)w_i/d},
 \ \ \
{\rm Spec}_{(0,1)}(t^d)=t^{\sum_iw_i} \cdot P(t)_{< d+1-\sum_iw_i}.$$
Hence, for any $0\leq \overline {l}\leq d+1-\sum_i w_i$ we have
$$ \hh(\overline{l})= \# \{ \mbox{$\alpha$  spectral number with $\alpha < (\overline{l}+\sum_i w_i)/d$} \},$$
and $\overline {w}_0(\overline{l})=\hh(\overline{l})+\hh(d+1-\sum_iw_i-\overline {l})-p_g$.

\begin{remark}
From the above discussion $\bH^{\geq 1}_{an}(X,o)=0$ and $\RR_{an}(X,o)$ is completely determined
by the Hodge spectrum ${\rm Spec}_{(0,1)}$.
 (Recall that $\RR_{an}(X,o)$ determines  $\bH^0_{an}(X,o)$ as well as its $\bH$.)

On the other hand, from $\RR_{an}(X,o)$ we cannot recover the precise values of the spectral numbers.
Indeed, e.g. for the Brieskorn (normal,  minimally elliptic surface) singularities
of type $(2,3,7)$ and $(2,3,11)$  have the same graded root, but their spectrum in $(0,1)$ are different.
Both have only one spectral number in $(0,1)$, they are $41/42$ and $61/66$ respectively.  On the other hand,
the graded root certainly provides interesting information about the mutual position of the spectral numbers.
\end{remark}
\begin{remark}
It would be very interesting to generalize this Hodge theoretical connection to all
hypersurfaces, and to find other reinterpretations of $\bH^*_{an}$ in terms of other classical analytic invariants (e.g.
the multiplicity).\end{remark}

\subsection{Newton nondegenerate  hypersurface singularities}\label{ss:NNN}

Assume that $f:(\C^{n+1}, 0)\to (\C,o)$ is an isolated hypersurface singularity
which is nondegenerate with respect to a convenient Newton diagram \cite{Kouch,Var}.
We also assume that there is no lattice point on the Newton boundary with all coordinates strictly positive. This is equivalent to the fact that there exists no
spectral number equal to 1 \cite{SaitoNN}.

The normal vectors of the top faces of the Newton diagram determine a dual fan. A regular subdivision of this fan determines a toric resolution $\tX\to X$ of $(X,o)=(\{f=0\},0)$ (even an embedded resolution,
but that one is not needed  here). Then by the toric correspondence, the normal directions of the top faces of the Newton diagram determine exceptional divisors in $\tX$, they are irreducible by the
nondegeneracy assumption.    Let their collection be $\overline{\calv}$. A very same proof as in the
weighted homogeneous case shows  that $\overline{\calv}$ is a B$_{an}$--set. One has again to consider
the lattice points below the diagram, their number is $p_g$, and they index both the spectral numbers
in the interval  $(0,1)$ and also differential forms forming a basis in
$H^0(\tX\setminus E, \Omega^n_{\tX})/H^0(\tX, \Omega^n_{\tX})$,
  \cite{MeTe,SaitoNN}.  The details are left to the reader.

  However, with respect to the complete discussion from \ref{ss:IHS},
  the parallelism with the weighted homogeneous germs breaks at some point:
  in the Newton nondegenerate case the combinatorial Newton filtration of the lattice points
  below the Newton diagram usually does not coincide with the corresponding divisorial  filtration.
  The description of $\bH^*_{an}$ and of $\RR_{an}$ is the subject of a forthcoming paper.


\begin{thebibliography}{30}

\bibitem{AgNe1} \'Agoston, T, and N\'emethi, A.: Analytic lattice cohomology of surface singularities,
arXiv:2108.12294.

\bibitem{AgNeIII} \'Agoston, T, and N\'emethi, A.: Analytic lattice cohomology of surface singularities II (equivariant case), arXiv:2108.12429.









\bibitem{BodNem} Bodn\'ar, J., N\'emethi, A.: Lattice cohomology and rational
cuspidal curves,  {\it Math. Research Letters} 23 (2016) no:2, 339--375.

\bibitem{BodNem2} Bodn\'ar, J., N\'emethi, A.:
Seiberg--Witten invariant of the universal abelian cover of $S^3_{-p/q}(K)$,
 Proceendings:
{\it Singularities and Computer Algebra -- Festschrift for Gert-Martin Greuel on the Occasion
of his 70ty Birthday}, Springer, 2016, Ed's: W. Decker, G. Pfister, M. Schulze.

\bibitem{BCG}  Bodn\'ar, J., Celoria, D. and Golla, M.: Cuspidal curves and Heegaard Floer homology,
{\it Proc. Lond. Math. Soc.} (3) {\bf 112} (2016), no. 3, 512--548.

\bibitem{BL1} Borodzik, M., Livingston, Ch.: Heegaard Floer homology and rational
cuspidal curves, {\it Forum Math. Sigma}  2 (2014), e28.



\bibitem{BLMN2}
de Bobadilla, J.F., Luengo, I., Melle-Hern\'andez, A.,
N\'emethi, A.: On rational cuspidal curves, open surfaces and
local singularities,
{\em Singularity theory,
Dedicated to Jean-Paul Brasselet on His 60th Birthday,}
Proc.~of the 2005 Marseille Singularity School and Conference,
2007, 411-442.

\bibitem{CDG} Campillo, A. and Delgado, F. and Gusein-Zade, S. M.:
   Poincar\'e series of a rational surface singularity,
Invent. Math., {\bf 155} (2004), 41-53.



\bibitem{CNP} Caubel, C., N\'emethi, A.,  Popescu--Pampu, P.:
Milnor open books and Milnor fillable contact 3-manifolds,
Topology {\bf 45}(3) (2006), 673-689.


\bibitem{CH}
  Colin, V.:
  R\'ealisations g\'eom\'etriques de l'homologie
  de Khovanov par des homologies de Floer (d'apr\'es
  Abouzaid-Seidel-Smith et Ozsv\'ath-Szab\'o). (French) [Geometric
    realizations of Khovanov homology by Floer homologies (following
    Abouzaid-Seidel-Smith and Ozsv\'ath-Szab\'o)] Ast\'erisque No. 367--368
  (2015), Exp. No. 1079, viii, 151--177.




\bibitem{CHR} Cutkosky, S.D, Herzog, J., Reguera, A.: Poincar\'e series of resolutions
of surface singularities,  Trans. of AMS, {\bf 356} (5) (2003), 1833--1874.

\bibitem{DadNem} D\u{a}d\u{a}rlat, M., N\'emethi, A.:
  Shape theory and (connective) K-theory,
J. Operator Theory, {\bf 23}(2) (1990), 207--291.

\bibitem{DoldThom} Dold, A.,  Thom, R. :
 Quasifaserungen und unendliche symmetrische Produkte,
  Annals of Mathematics, Second Series, {\bf 67} (1958),  239--281.







\bibitem{Greene} Greene, J. E.: A surgery triangle for lattice cohomology,
Algebraic \& Geometric Topology {\bf 13} (2013), 441-451.

\bibitem{GuRo} Gunning, R.  and Rossi, H.: Analytic functions of several variables,
Englewood Cliffs, N.J.:
Prentice-Hall 1965.

\bibitem{GrRie} Grauert, H. and Riemenschneider, O.: Verschwindungss\"atze f\"ur analytische
kohomologiegruppen auf komplexen R\"aumen, {\it Inventiones math.} {\bf 11} (1970), 263-292.

\bibitem{Hartshorne}  Hartshorne, R:
	Algebraic Geometry, 		Graduate Texts in Mathematics
		{\bf 52} 1977,  Springer-Verlag.

\bibitem{HR} Hochster, M. and  Roberts, J.:
 Rings of invariants of reductive groups acting on regular rings are Cohen-Macaulay,
 {\it Adv. Math.} {\bf  13 } (1974), 115--175.


\bibitem{Ishii} Ishii, S.:  On Isolated Gorenstein Singularities,
{\it Math. Ann.} {\bf 270} (1985), 541--554.

\bibitem{Karras} Karras, U.: Local Cohomology Along Exceptional Sets, {\it Math. Ann.} {\bf
275} (1086), 673--682.



\bibitem{Kouch} Kouchnirenko,  A. G.:
Poly\`edres de Newton et nombres de Milnor, {\it Invent. Math.} {\bf  32}
(1976), no. 1, 1-31.



\bibitem{LN1} L\'aszl\'o, T. and N\'emethi, A.:
Reduction theorem for lattice cohomology,
{\it Int. Math. Res. Notices } {\bf 2015}, Issue 11 (2015), 2938--2985.


\bibitem{Laufer72} Laufer, H.B.: On rational singularities,
{\em Amer. J. of Math.}, {\bf 94}, 597-608, 1972.





\bibitem{Lem} Lemahieu A.: Poincar\'e series of embedded filtrations,
		 {\it Mathematical Research Letters},
		{\bf 18}{5} (2011), 815--825.

\bibitem{MeTe} Merle, M., Teissier, B.: Conditions d'adjonction, d'apr\`es Du Val,
{\it S\'eminaire sur les singularit\'es des surfaces (Polytechnique) (1976-1977)}, exp. no {\bf 14},
 1--16.







\bibitem{NOSz} N\'emethi, A.: On the Ozsv\'ath--Szab\'o invariant of negative definite plumbed 3-manifolds,
{\it Geometry \& Topology}, {\bf 9}(2)(2005), 991--1042.

\bibitem{NSurgd} N\'emethi, A.:  On the Heegaard Floer homology of $S^3_{-d}(K)$ and
unicuspidal rational plane curves,  {\em Fields Institute
Communications}, Vol. {\bf 47}, 2005, 219-234;
 ``Geometry and Topology of Manifolds'',
Eds: H.U. Boden, I. Hambleton, A.J. Nicas and B.D. Park,

\bibitem{NGr} N\'emethi, A.: Graded roots and singularities,
 Proc. {\em Advanced School and Workshop} on {\em Singularities in Geometry and
Topology} ICTP (Trieste, Italy), World Sci. Publ., Hackensack, NJ, 2007, 394-463.

\bibitem{Nlattice} N\'emethi, A.: Lattice cohomology of normal surface
  singularities.   Publ. Res. Inst. Math. Sci.  {\bf 44} (2)
  (2008),  507-543.


\bibitem{NJEMS} N\'emethi, A.: The Seiberg--Witten invariants of negative definite plumbed 3--manifolds,
{\em J. Eur. Math. Soc.} {\bf 13} (2011), 959--974.


\bibitem{Nexseq} N\'emethi, A.: Two exact sequences for lattice cohomology,
Proc. of the conference  to honor H. Moscovici's 65th birthday,
{\it Contemporary Math.} {\bf 546} (2011), 249--269.

\bibitem{NeLO} N\'emethi, A.:
Links of rational singularities, L-spaces and LO fundamental groups,
 Inventiones mathematicae {\bf 210}(1) (2017), 69--83.


\bibitem{NeNi} N\'emethi, A. and Nicolaescu, L.I.:  Seiberg-Witten
invariants and surface singularities,
Geometry and Topology, Volume {\bf 6} (2002), 269-328.







\bibitem{NSig} N\'emethi, A. and Sigur{\dh}sson, B.:
	The geometric genus of hypersurface singularities,
		{\it Journal of the Eur. Math. Soc.},
		{\bf 18}(4) (2016), 825--851.

\bibitem{NSigNN} N\'emethi, A. and Sigur{\dh}sson, B.:
Local Newton nondegenerate Weil divisors in toric varieties, arXiv
2102.02948 (2021).
	

\bibitem{OSz} Ozsv\'ath, P. and  Szab\'o, Z.:
 Holomorphic discs
and topological invariants for closed  three-spheres, Ann. of
Math. (2) 159 (2004), no. 3, 1127--1158.

\bibitem{OSz7}  Ozsv\'ath, P. and  Szab\'o, Z.: Holomorphic discs
and three-manifold invariants: properties and applications, Ann.
of Math. (2) 159 (2004), no. 3, 1159--1245.


\bibitem{OSzP}  Ozsv\'ath, P.S. and  Szab\'o, Z.: On the Floer
homology of plumbed three-manifolds,  Geom. Topol., 7 (2003),
185--224.

\bibitem{OSSz3} Ozsv\'ath, P.; Stipsicz, A. I. and  Szab\'o, Z.:
  A spectral sequence on lattice homology,
  Quantum Topol. {\bf5} (2014),  487--521.





\bibitem{MR}  Reid, M.: Chapters on Algebraic Surfaces.
In: Complex Algebraic Geometry,
IAS/Park City Mathematical Series,  Volume {\bf 3}  (J. Koll\'ar editor),
3-159, 1997.


\bibitem{SaitoNN} Saito, M.: Exponents and Newton polyhedra of isolated hypersurface singularities,
Math. Ann. {\bf 281} (3), (1988), 411--417.

\bibitem{BaldurFiltr} Sigur{\dh}sson, B.: On ideal filtrations for Newton nondegenerate surface singularities,
		arXiv 1911.00095, (2019).

\bibitem{StOslo} Steenbrink, J.H.M.: Mixed Hodge structures of the vanishing cohomology,
Nordic Summer School/NAVF, Symp. in Math., Oslo, August 5-25, 1976, 525-563.

\bibitem{SteenW} Steenbrink, J.: Intersection form for quasi-homogeneous singularities,
{\it Comp. Math.} {\bf 34} (2) (1977), 211--223.

 \bibitem{Steen} Steenbrink, J. H. M.: Mixed Hodge structures associated with isolated singularities,
 Proc. of Symposia in Pure Math. {\bf 40}, part 2, 1983, 513--536.

\bibitem{Ust}  Ustilovsky, I.: Infinitely many contact structures on
$S^{4m+1}$,  Internat. Math. Res. Notices 1999, no. 14, 781--791.

\bibitem{Var} Varchenko, A. N.: Zeta-function of monodromy and Newton's diagram, {\it Invent. Math.} {\bf 37} (1976), 253--262.


\bibitem{Vi} Viehweg, E.: Rational singularities of higher dimsnional schemes,
{\it Proc. AMS} {\bf 63} (1) (1977), 6--8.


\bibitem{Wlod} W{\l}odarczyk, J.: Decomposition of birational toric maps in blow-ups \& blow-downs, {\it  Trans. Amer. Math. Soc.} {\bf  349} (1997), no. 1, 373–-411.

\bibitem{YauTwo} Yau, S. S.-T.: Two Theorems on Higher Dimensional Singularities,
{\it Math. Annalen} {\bf 231} (1977), 55--59.





\end{thebibliography}
\end{document}